\documentclass{amsart}

\usepackage{array,booktabs,comment,delarray,float,rotating,titlesec}
\usepackage{amsmath,amsfonts,amssymb,amsthm}
\usepackage{bbm,bm,cases,dsfont,extarrows,mathrsfs,mathtools}
\usepackage{graphicx,xcolor}
\usepackage{pb-diagram,tikz,tikz-cd}
\usetikzlibrary{arrows,graphs}
\usepackage[all]{xy}
\usepackage[numbers]{natbib}
\usepackage{hyperref}

\makeatletter
\providecommand\@dotsep{4.5}
\providecommand\numberline[1]{\hb@xt@\@tempdima{#1\hfil}}
\renewcommand*\l@section{\@dottedtocline{1}{1.5em}{3em}}
\renewcommand*\l@subsection{\@dottedtocline{2}{3.0em}{3.8em}}
\makeatother

\usepackage{titlesec}
\titleformat{\section}{\normalfont\bfseries\Large}{\thesection}{1em}{}
\titlespacing*{\section}{0pt}{3.5ex plus 1ex minus .2ex}{2.3ex plus .2ex}
\titleformat{\subsection}{\normalfont\bfseries\large}{\thesubsection}{1em}{}
\titlespacing*{\subsection}{0pt}{2.5ex plus 1ex minus .2ex}{1.5ex plus .2ex}

\newtheorem{theorem}{Theorem}[section]
\newtheorem{lemma}[theorem]{Lemma}
\newtheorem{definition}[theorem]{Definition}
\newtheorem{corollary}[theorem]{Corollary}
\newtheorem{proposition}[theorem]{Proposition}
\newtheorem{example}[theorem]{Example}
\newtheorem{remark}[theorem]{Remark}

\renewcommand{\hom}{\operatorname{Hom}}
\newcommand{\ext}{\operatorname{Ext}}
\renewcommand{\mod}{\bmod}
\newcommand{\sstilt}{\operatorname{s\tau\text{-}tilt}}
\newcommand{\ftors}{\operatorname{f\text{-}tors}}
\newcommand{\tors}{\operatorname{tors}}
\newcommand{\mut}{\operatorname{Mut}}
\newcommand{\proj}{\operatorname{proj}}
\newcommand{\inj}{\operatorname{inj}}
\newcommand{\add}{\operatorname{add}}
\newcommand{\filt}{\operatorname{Filt}}
\newcommand{\fac}{\operatorname{Fac}}
\newcommand{\sub}{\operatorname{Sub}}

\newcommand{\cA}{\mathcal{A}}
\newcommand{\cC}{\mathcal{C}}
\newcommand{\cS}{\mathcal{S}}
\newcommand{\cT}{\mathcal{T}}
\newcommand{\cF}{\mathcal{F}}

\newcommand{\cJ}{\mathcal{J}}
\renewcommand{\ker}{{\rm Ker}}
\newcommand{\coker}{{\rm Coker}}
\newcommand{\im}{{\rm Im}}
\newcommand{\ideal}{\mathbf{Id}}
\newcommand{\filter}{\mathbf{Filter}}
\newcommand{\rad}{\mbox{rad}}

\newcommand{\bZ}{{\mathbb Z}}

\newcommand{\lra}{\longrightarrow}

\newcommand{\ra}{\rightarrow}
\newcommand{\sdp}{\times\kern-.2em\vrule height1.1ex depth-.05ex}
\newcommand{\epi}{\lra \kern-.8em\ra}

\makeatletter
\@addtoreset{equation}{section}
\makeatother

\title[Theory of Infinite Maximal Green Sequences]{Theory of Infinite Maximal Green Sequences: Finest Stability Data, CBHO Sequence and Cluster quiver pattern}

\author{Fang Li}
\thanks{Fang Li\qquad \texttt{fangli@zju.edu.cn}\\School of Mathematical Sciences, Zhejiang University, Yuhangtang Road 866, Hangzhou, Zhejiang 310058, China P.R.}

\author{Kangping Qu}
\thanks{Kangping Qu\qquad \texttt{qukp.zju@gmail.com}\\School of Mathematical Sciences, Zhejiang University, Yuhangtang Road 866, Hangzhou, Zhejiang 310058, China P.R.}
\date{\today}

\begin{document}
	
	\begin{abstract}
		In this paper, we introduce maximal green sequences of infinite length in abelian length categories and establish natural bijections among infinite maximal green sequences, finest stability data, and complete backward $\operatorname{Hom}$-orthogonal (briefly, CBHO) sequences of bricks.
		
		For cluster quiver patterns (equivalently, cluster algebras) of infinite rank, we define maximal green sequences of infinite length via $c$-vectors. Moreover, we prove that under suitable conditions such a maximal green sequence in a cluster quiver pattern of infinite rank can be categorified as a maximal green sequence of torsion classes in the corresponding representation category via a specific construction, thereby providing an example of a maximal green sequence in an abelian length category. This framework establishes a connection between infinite maximal green sequences in representation theory and maximal green sequences for cluster algebras of infinite rank.
	\end{abstract}
	
	\maketitle
	
	{\small
		
		\medskip
		
	}
	
	\tableofcontents
	
	\section{Introduction}
	
	Maximal green sequences were firstly introduced by Keller in \cite{K} and are defined as particular sequences of mutations of framed cluster quivers. Since then, they have found remarkable applications in diverse areas: they appear in the counting of BPS states in string theory, give rise to quantum dilogarithm identities, and play a key role in computing refined Donaldson–Thomas invariants.
	
	Through categorification, maximal green sequences can be naturally interpreted in representation theory. For instance, given a cluster quiver, a maximal green sequence is in bijection with a finite chain of torsion classes. This perspective enables the study of maximal green sequences in module categories and, more generally, in arbitrary abelian categories.
	
	Stability conditions first appeared in geometric invariant theory in the work of Mumford \cite{Mu}, where they were used to construct moduli spaces of vector bundles on algebraic curves. King \cite{Ki} later linked stability conditions in geometric invariant theory to the representation theory of finite-dimensional algebras. Bridgeland studied the stability conditions in triangulated categories \cite{Br}, while Rudakov studied them in abelian categories \cite{Ru}. A key feature of stability conditions is the Harder–Narasimhan filtration, which can be defined axiomatically in abelian categories; this leads to the notion of stability data introduced in \cite{GKR}.
	
	In the finite setting, the relationships among these notions have been extensively studied. Br\"ustle, Smith, and Treffinger \cite{BST2} showed that finite chains of torsion classes correspond precisely to finite finest stability data. Chen, Lin, and Ruan \cite{CLR} gave a criterion for a stability data to be finest, and Liu and Li \cite{LL} studied when a maximal green sequence can be induced by a central charge.
	
	In the finite setting, Igusa \cite{Ig2} established a one-to-one correspondence among finite maximal green sequences, finite finest stability data, and complete forward $\hom$-orthogonal sequences of bricks (which are essentially the same as backward sequences up to reversing order). In this paper, we introduce the notion of maximal green sequences of infinite length and extend this equivalence to arbitrary abelian length categories, allowing maximal green sequences of infinite length. Our first main result extends Igusa's work by establishing the following three-way equivalence.
	
	\begin{theorem}[Theorem \ref{thm:correspondence of three sets}]\label{thm:main A}
		Let $\mathcal{A}$ be an abelian length category. There are natural bijections among the three sets:
		\begin{enumerate}
			\item the set of finest stability data on $\mathcal{A}$;
			\item the set of maximal green sequences in $\mathcal{A}$ (including those of infinite length);
			\item the set of complete backward $\hom$‑orthogonal (CBHO) sequences of bricks in $\mathcal{A}$.
		\end{enumerate}
	\end{theorem}
	
	Cluster algebras, introduced by Fomin and Zelevinsky, are commutative algebras with a distinguished set of generators called cluster variables\cite{FZ1}; maximal green sequences play a central role in their mutation theory. Cao and Li \cite{CL} previously proved that in finite rank, if a finite quiver is a triangular extension of two finite quivers, then a maximal green sequence for each component gives a maximal green sequence for the whole quiver. This result provides a useful tool for constructing maximal green sequences.
	
	Cluster algebras of infinite rank have been studied from various perspectives. Gratz \cite{G}  showed that any cluster algebra of infinite rank can be expressed as a colimit of finite rank ones, which provides a framework for extending the results on finite rank to the infinite setting. Huang and Li \cite{HL} extended this colimit idea to a categorical framework by considering the so-called strongly almost finite quivers. The present work is inspired by the colimit perspective developed in \cite{G} and then in \cite{HL}.
	
	In this paper, we work with strongly almost finite quivers and define maximal green sequences of infinite length. We extend the above result to this infinite setting as follows.
	
	\begin{theorem}[Theorem \ref{thm:maximal green sequence and triangular extension}]\label{thm:main B}
		Suppose $Q$ is a strongly almost finite quiver and is a triangular extension of the sequence of finite quivers $\{Q_p\}_{p=1}^\infty$. For each $p\ge 1$, let $\mathbf{i}_p=(i_{p1},\dots,i_{pk_p})$ be a maximal green sequence of $Q_p$. Then
		\begin{equation}
			\mathbf{i}=(\mathbf{i}_1; \mathbf{i}_2; \cdots )
		\end{equation}
		is a maximal green sequence of $Q$.
	\end{theorem}
	
	Theorem \ref{thm:main B} shows how to assemble infinite maximal green sequences from finite ones. To understand the representation-theoretic content of such sequences, we turn to categorification. Our third main theorem provides a categorical realization in the module category $\mod KQ$.
	
	\begin{theorem}[Theorem \ref{thm:categorification of mgs}]\label{thm:main C}
		Let $Q$ be a strongly almost finite quiver and $\mathbf{i} = (i_1, i_2, \dots)$ be a maximal green sequence of $Q$ satisfying two natural conditions. Then there exists a unique maximal green sequence $\{\cT_k\}_{k\ge 0}$ of torsion classes in $\mod KQ$ such that for each $k \ge 1$,
		\begin{equation}
			\cT_{k-1}^\perp \cap \cT_k = \add B_k,
		\end{equation}
		where $B_k$ is a rigid brick whose dimension vector equals the $c$-vector associated to the mutation at $i_k$.
	\end{theorem}
	
	This theorem bridges combinatorial $c$-vectors and representation-theoretic data. In particular, it yields a converse to Theorem \ref{thm:main B}: under suitable conditions, a maximal green sequence of a triangular extension of infinite quivers decomposes into maximal green sequences of its finite components. More precisely, we have the following corollary.
	
	\begin{corollary}[Corollary \ref{cor:triangulation and mgs}]
		Let $Q$ be a strongly almost finite quiver which is a triangular extension of the sequence of finite quivers $\{Q_p\}_{p=1}^\infty$. Let $\mathbf{k} = \mathbf{k}_1 \ast \mathbf{k}_2 \ast \cdots$ be a mutation sequence of $Q$ where each $\mathbf{k}_p$ is a sequence of vertices of $Q_p$. Assume that $\mathbf{k}$ satisfies conditions (1) and (2) of Theorem~\ref{thm:categorification of mgs}. Then $\mathbf{k}$ is a maximal green sequence of $Q$ if and only if each $\mathbf{k}_p$ is a maximal green sequence of $Q_p$.
	\end{corollary}
	
	This paper is organized as follows. In Section~\ref{sec:preliminary}, we recall the necessary background on abelian length categories, torsion pairs, stability data, and lattice theory. Section~\ref{sec:correspondence} is devoted to the proof of Theorem \ref{thm:main A}, establishing the bijections among finest stability data, maximal green sequences (possibly of infinite length), and complete backward $\hom$-orthogonal sequences of bricks. In Section~\ref{sec:cluster}, we turn to cluster algebras of infinite rank: we first introduce the notion of maximal green sequences for strongly almost finite quivers and prove Theorem \ref{thm:main B}, which shows how to assemble infinite maximal green sequences from finite ones; we then provide a categorical realization of such sequences in the module category $\mod KQ$, proving Theorem \ref{thm:main C} and thereby establishing a converse to Theorem \ref{thm:main B}.
	
	\subsection*{Notation}
	Throughout this paper, $\cA$ denotes a $K$-linear Hom-finite small abelian length category over an algebraically closed field $K$. For a subcategory $\mathcal{X}\subseteq\cA$, $\add\mathcal{X}$ is the subcategory of finite direct sums of objects in $\mathcal{X}$, $\fac\mathcal{X}$ consists of quotients of objects in $\mathcal{X}$, $\sub\mathcal{X}$ consists of subobjects of objects in $\mathcal{X}$, and $\filt(\mathcal{X})$ is the smallest extension-closed subcategory containing $\mathcal{X}$ (i.e., objects admitting a finite filtration with composition factors in $\mathcal{X}$). For $\mathcal{X}\subseteq\cA$, $\prescript{\perp}{}{\mathcal{X}}=\{Y\in\cA\mid\hom_{\cA}(Y,X)=0,\ \forall X\in\mathcal{X}\}$ and $\mathcal{X}^{\perp}=\{Y\in\cA\mid\hom_{\cA}(X,Y)=0,\ \forall X\in\mathcal{X}\}$ denote the left and right orthogonal subcategories, respectively.
	
	\section{Preliminary}\label{sec:preliminary}
	In this section, we will recall some basic concepts in abelian categories and lattice theory.
	
	\subsection{ On abelian categories}
	Let $X\in \cA$, we say that $X$ has \textbf{finite length} if there exists a finite filtration
	\begin{equation}
		0=X_0\subsetneq X_1\subsetneq X_2\subsetneq \cdots \subsetneq X_n=X
	\end{equation}
	such that $X_i/X_{i-1}$ is a simple object in $\cA$. By the \textbf{Jordan-H\"older theorem}, whenever the filtration above exists, each simple module (composition factor) is uniquely determined by $X$ up to isomorphism and multiplicity. We say that $\cA$ is an \textbf{abelian length category} if each object in $\cA$ has finite length. For example, if $\Lambda$ is a finite dimensional algebra over a field $K$, the the category $\mod \Lambda$ of finitely generated $\Lambda$-modules is an abelian length category.
	
	\begin{definition}
		Suppose $\cA$ is an abelian category, and $\cT$, $\cF$ are full subcategories of $\cA$. Then $(\cT, \cF)$ is called a torsion pair if the following two conditions are satisfied:
		\begin{enumerate}
			\item $\hom_{\cA}(X, Y)=0$ for all $X\in \cT$ and $Y\in \cF$.
			\item for any object $X\in \cA$, there exists a short exact sequence
			\begin{equation}\label{eq:canonical sequence}
				0\rightarrow tX\rightarrow X\rightarrow X/tX\rightarrow 0
			\end{equation}
			such that $tX\in \cT$ and $X/tX\in \cF$.
		\end{enumerate}
		In this case, we say that $\cT$ is a \textbf{torsion class} and $\cF$ is a \textbf{torsion-free class}.
	\end{definition}
	
	The exact sequence \eqref{eq:canonical sequence} is unique up to isomorphism and is called the \textbf{canonical sequence} of $X$ with respect to the torsion pair $(\cT, \cF)$. It is well-known that a subcategory $\cT$ is a torsion class of some torsion pair if and only if $\cT$ is closed under quotients and extensions. Dually, a subcategory $\cF$ is a torsion-free class of some torsion pair if and only if $\cF$ is closed under subobjects and extensions. Moreover, a pair of subcategories $(\cT, \cF)$ is a torsion pair in $\cA$ if and only if $\cF = \cT^{\perp}$ and $\cT = \prescript{\perp}{}{\cF}$.
	
	Stability data were introduced in \cite{GKR} as a generalization of the stability function considered in \cite{BST2}. A stability data allows every object to be filtered by several semistable objects.
	
	\begin{definition}\cite[Definition 2.4]{GKR}
		Suppose that $\mathcal{A}$ is an abelian category, $\Phi$ is a totally ordered set, and a nonzero extension-closed subcategory $\Pi_{\varphi} \subset \mathcal{A}$ is given for every $\varphi \in \Phi$. A pair $\left(\Phi,\left\{\Pi_{\varphi}\right\}_{\varphi \in \Phi}\right)$ is called a \textbf{stability data} if
		
		\begin{enumerate}
			\item $\hom_{\mathcal{A}} \left(\Pi_{\varphi^{\prime}}, \Pi_{\varphi^{\prime \prime}}\right)=0$ for all $\varphi^{\prime}>\varphi^{\prime \prime}$ in $\Phi$;
			\item every non-zero object $X \in \mathcal{A}$ admits a unique filtration called \textbf{Harder-Narasimhan filtration} (HN-filtration for short):
			\begin{equation}\label{eq:HN-filtration}
				\begin{tikzcd}
					{0=X_0} & {X_1} & {X_2} & \cdots & {X_{n-1}} & {X_n = X} \\
					& {A_1} & {A_2} && {A_{n-1}} & {A_n}
					\arrow[hook, from=1-1, to=1-2]
					\arrow[hook, from=1-2, to=1-3]
					\arrow[two heads, from=1-2, to=2-2]
					\arrow[hook, from=1-3, to=1-4]
					\arrow[two heads, from=1-3, to=2-3]
					\arrow[hook, from=1-4, to=1-5]
					\arrow[hook, from=1-5, to=1-6]
					\arrow[two heads, from=1-5, to=2-5]
					\arrow[two heads, from=1-6, to=2-6]
				\end{tikzcd}
			\end{equation}
			with non-zero factors $A_i=X_i / X_{i-1} \in \Pi_{\varphi_i}$ and strictly decreasing $\varphi_{i}>\varphi_{i+1}$ for any $1\leq i\leq n-1$.
		\end{enumerate}
		
	\end{definition}
	
	\begin{remark}
		For convenience, we require that $\Pi_{\varphi}\ne \left\{0\right\}$ for any $\varphi\in \Phi$. It will not cause any problems essentially since $\left\{0\right\}$ makes no contributions to the filtration of an object $X$ and then we can delete them and we finally obtain a same HN-filtration.  Our definition is a little different from that in \cite{CLR} and the slicing considered in \cite{Tr}.
	\end{remark}
	
	The categories $\Pi_{\varphi}$ are called the \textbf{semistable subcategories} of the stability data $\left(\Phi,\left\{\Pi_{\varphi}\right\}_{\varphi \in \Phi}\right)$. The nonzero objects in $\Pi_{\varphi}$ are said to be \textbf{semistable of phase} $\varphi$, while the simple objects in $\Pi_{\varphi}$ are said to be \textbf{stable}. The filtration (2.1) is called the  \textbf{Harder-Narasimhan filtration (HN-filtration for short)} of $X$, which is unique up to isomorphism. The factors $A_i$ 's are called the \textbf{HN-factors} of $X$.
	
	A non-zero semistable subobject $Y$ of $X$ is called the \textbf{maximal semistable subobject}, if any non-zero semistable subobject $Y'$ of $X$ satisfies that the phase of $Y'$ is not greater than that of $Y$, and $Y' \subseteq Y$ when they have the same phase. Dually, a non-zero semistable quotient object $Z$ of $X$ is called the \textbf{minimal semistable quotient object}, if any non-zero semistable quotient object $Z'$ of $X$ satisfies that the phase of $Z'$ is not less than that of $Z$, and $Z'$ is a quotient of $Z$ when they have the same phase. By definition, we know that in \eqref{eq:HN-filtration}, $A_1=X_1 / X_0$ is the maximal semistable subobject of $X=X / X_0$, and $A_n$ is the minimal semistable quotient object of $X=X_n$. Define $\boldsymbol{\phi}^{+}(X):=\varphi_1$ and $\phi^{-}(X):=\varphi_n$. Then $X \in \Pi_{\varphi}$ if and only if $\phi^{+}(X)=\phi^{-}(X)=\varphi=: \phi(X)$.
	
	We have the following observation:
	
	\begin{lemma}\label{lem:phase in exact sequence}
		Suppose $0\rightarrow A\rightarrow B\rightarrow C\rightarrow 0$ is an exact sequence in the abelian category $\cA$, then:
		
		\begin{enumerate}
			\item $\phi^+(A)\leq \phi^+(B)$ and $\phi^-(C)\ge \phi^-(B)$;
			\item $\phi^+(B)\leq \max\left\{\phi^+(A),\phi^+(C)\right\}$ and $\phi^-(B)\ge \min\left\{\phi^-(A),\phi^-(C)\right\}$.
		\end{enumerate}
	\end{lemma}
	
	\begin{proof}
		(1) is obvious from the definition. To prove (2), consider the HN-filtration of $C$:
		\begin{equation}
			0=C_0\subset C_1\subset C_2\subset \cdots \subset C_n=C
		\end{equation}
		and the pullback of $0\rightarrow A\rightarrow B\rightarrow C\rightarrow 0$ along the inclusion morphism $\iota_n:C_{n-1}\rightarrow C_n$:
		\begin{equation}
			\begin{tikzcd}
				0 & A & {B_{n-1}} & {C_{n-1}} & 0 & {} \\
				0 & A & B & C & 0
				\arrow[from=1-1, to=1-2]
				\arrow[from=1-2, to=1-3]
				\arrow[equals, from=1-2, to=2-2]
				\arrow[from=1-3, to=1-4]
				\arrow["{s_n}", from=1-3, to=2-3]
				\arrow[from=1-4, to=1-5]
				\arrow["{\iota_n}", from=1-4, to=2-4]
				\arrow[from=2-1, to=2-2]
				\arrow[from=2-2, to=2-3]
				\arrow[from=2-3, to=2-4]
				\arrow[from=2-4, to=2-5]
			\end{tikzcd}
		\end{equation}
		We have $\coker s_n\cong \coker \iota_n$. Repeat this procedure, we know that 
		\begin{equation}
			B\in \filter\left\{\Pi_{\varphi}\ |\ \min\left\{\phi^-(A),\phi^-(C)\right\}\leq \varphi\leq \max\left\{\phi^+(A),\phi^+(C)\right\}\right\}
		\end{equation}
		and (2) follows by definition.
	\end{proof}
	
	Given an abelian category $\cA$, there is a partial order on the set of all stability data on $\cA$.
	
	\begin{definition}\cite{CLR}
		Let $\left(\Phi,\left\{\Pi_{\varphi}\right\}_{\varphi\in\Phi}\right)$ and $\left(\Psi,\left\{P_{\psi}\right\}_{\psi\in\Psi}\right)$ be two stability data on an abelian category $\mathcal{A}$.
		\begin{enumerate}
			\item They are called \textbf{equivalent} if there exists an order-preserved bijective map $r: \Phi \rightarrow \Psi$ such that $P_{r(\varphi)} = \Pi_{\varphi}$ for any $\varphi \in \Phi$.
			\item We say that the stability data $\left(\Psi,\left\{P_{\psi}\right\}_{\psi\in\Psi}\right)$ is \textbf{finer} than $\left(\Phi,\left\{\Pi_{\varphi}\right\}_{\varphi\in\Phi}\right)$, or the latter is \textbf{coarser} than the former, and write $\left(\Phi,\left\{\Pi_{\varphi}\right\}_{\varphi\in\Phi}\right) \preceq \left(\Psi,\left\{P_{\psi}\right\}_{\psi\in\Psi}\right)$, if there exists a surjective map $r: \Psi \rightarrow \Phi$ satisfying:
			\begin{enumerate}
				\item $\psi^{\prime} > \psi^{\prime\prime}$ implies $r(\psi^{\prime}) \geq r(\psi^{\prime\prime})$;
				\item for any $\varphi \in \Phi$, $\Pi_{\varphi} = \filt\{P_{\psi} \mid \psi \in r^{-1}(\varphi) \}$.
			\end{enumerate}
		\end{enumerate}
	\end{definition}
	
	The relation ``finer-coarser'' defines a partial order on the set of all stability data on $\cA$. The minimal elements with respect to this partial order is called the \textbf{finest} stability data. The authors in \cite{CLR} have provided a criterion for a stability data to be finest on arbitrary abelian category.
	
	\begin{theorem}\cite{CLR}\label{thm:finest stability data}
		Let $\mathcal{A}$ be an abelian category. A stability data $\left(\Phi,\left\{\Pi_{\varphi}\right\}_{\varphi \in \Phi}\right)$ on $\cA$ is finest if and only if for any $\varphi\in \Phi$ and any non-zero objects $X,Y\in \Pi_{\varphi}$ , $\hom_{\cA}(X, Y)\ne 0\ne \hom_{\cA}(Y,X)$.
	\end{theorem}
	
	\subsection{On lattice theory}
	
	In this section we recall the lattice‑theoretic notions that will be used throughout the paper, especially in the study of torsion classes and maximal green sequences. For a comprehensive treatment we refer to \cite{DP}.
	
	\begin{definition}
		A \textbf{partially ordered set} (or \textbf{poset}) is a set $P$ equipped with a binary relation $\le$ that is reflexive, anti-symmetric and transitive.
	\end{definition}
	
	\begin{definition}
		A poset $L$ is a \textbf{lattice} if every pair of elements $x,y\in L$ has a supremum (join) $x\vee y$ and an infimum (meet) $x\wedge y$. A lattice $L$ is \textbf{complete} if every subset $S\subseteq L$ has a supremum $\bigvee S$ and an infimum $\bigwedge S$ in $L$ (indeed, the existence of all suprema implies the existence of all infima, and vice versa).
	\end{definition}
	
	In a complete lattice $L$ we denote the least element by $\bot_L$ and the greatest element by $\top_L$. If there is no ambiguity, we simply write them as $\bot$ and $\top$. For the empty set, we adopt the convention $\bigvee \emptyset = \bot$ and $\bigwedge \emptyset = \top$.
	
	\begin{definition}
		\begin{enumerate}
			\item An element $c$ of a complete lattice $L$ is \textbf{compact} if for every $S\subseteq L$ with $c\le \bigvee S$, there exists a finite subset $S_0\subseteq S$ such that $c\le \bigvee S_0$.  The lattice $L$ is called \textbf{algebraic} if every element is the join of some compact elements.
			\item Dually, an element $d$ of $L$ is \textbf{co-compact} if for every $S\subseteq L$ with $\bigwedge S\le d$, there exists a finite subset $S_0\subseteq S$ such that $\bigwedge S_0\le d$. The lattice $L$ is called \textbf{co-algebraic} if every element is the meet of some co-compact elements.
		\end{enumerate}
		If $L$ is both algebraic and co-algebraic, it is called \textbf{bialgebraic}.
	\end{definition}
	
	\begin{definition}
		Let $L$ be a complete lattice.
		
		(1) An element $x\in L$ is \textbf{completely join‑irreducible} if whenever $x=\bigvee S$ for some $S\subseteq L$, then $x\in S$.
		
		(2) An element $x\in L$ is \textbf{completely meet‑irreducible} if whenever $x=\bigwedge S$ for some $S\subseteq L$, then $x\in S$.
	\end{definition}
	
	Note that $\bot$ is not completely join-irreducible and $\top$ is not completely meet-irreducible since $\bigvee \emptyset = \bot$ and $\bigwedge \emptyset = \top$.
	
	\begin{definition}
		Let $P$ be a partially ordered set.
		\begin{enumerate}
			\item A non-empty subset $I\subseteq P$ is an \textbf{ideal} if it is downward closed and directed: for any $x,y\in I$ there exists $z\in I$ with $x,y\le z$.
			\item A non-empty subset $F\subseteq P$ is a \textbf{filter} if it is upward closed and filtered: for any $x,y\in F$ there exists $z\in F$ with $z\le x,y$.
		\end{enumerate}
		Denote by $\ideal(P)$ and $\filter(P)$ the sets of ideals and filters, and set $\ideal_0(P)=\ideal(P)\cup\{\emptyset\}$, $\filter_0(P)=\filter(P)\cup\{\emptyset\}$.
		
		For $p\in P$, $\downarrow p = \{x\mid x\le p\}$ and $\uparrow p = \{x\mid p\le x\}$ are the \textbf{principal ideal} and \textbf{principal filter} generated by $p$. For $A\subseteq P$, $\downarrow A$ (resp. $\uparrow A$) is the smallest ideal (resp. filter) containing $A$, given by $\{x\mid \exists a\in A,\ x\le a\}$ (resp. $\{x\mid \exists a\in A,\ a\le x\}$).
	\end{definition}
	
	\begin{remark}
		In the next section, we will mainly consider filters on totally ordered sets. In that case, a subset is a filter precisely when it is upward closed.
	\end{remark}
	
	The following proposition shows that every complete algebraic lattice can be recovered from its poset of compact elements. In the next section, we only need a special case of this proposition. However, we'll prove the general version anyway.
	
	\begin{proposition}\label{prop:realization of ideal}
		Let $L$ be a complete algebraic lattice, then there exists a unique (up to isomorphism) poset $P$ such that $L\cong \ideal(P)$ as a complete lattice. To be more precise, $P$ can be taken as $K(L)$ and the map 
		\begin{equation}
			\phi:L\rightarrow \ideal(K(L)), \quad \phi(x)=\{c\in K(L)\ |\ c\leq x\}
		\end{equation}
		is a complete lattice isomorphism.
	\end{proposition}
	
	\begin{proof}
		The proof of existence consists of three parts.
		
		Firstly, for any $x \in L$, the set $\phi(x)$ is closed downward. For any $c_1, c_2 \in \phi(x)$, their join $c_1 \vee c_2$ is compact and satisfies $c_1 \vee c_2 \leq x$, hence $c_1 \vee c_2 \in \phi(x)$. Therefore, $\phi(x)$ is an ideal.
		
		Next we show that $\phi$ is a complete lattice homomorphism. By construction, $\phi$ is order-preserving. Note that the meet of ideals is given by the intersection of sets, so the meet preservation is immediate. For join-preservation, consider $S \subset L$. We need to verify the middle equality in 
		\begin{equation}
			\phi\left(\bigvee S\right) 
			= \left\{ c \in K(L) \mid c \leq \bigvee S \right\}
			= \bigvee_{x \in S} \left\{ c \in K(L) \mid c \leq x \right\}
			= \bigvee \phi(S).
		\end{equation}
		The containment ``$\supseteq$'' is obvious. For ``$\subseteq$'', suppose $c \leq \bigvee S$. By the algebraicity of $L$ and compactness of $c$, there exists a finite set $\{c_1,\ldots,c_n\}$ of compact elements such that $c \leq c_1 \vee \cdots \vee c_n$, where each $c_i \leq x_i$ for some $x_i \in S$. This establishes the required containment.
		
		Finally, we prove that $\phi$ is bijective. If $\phi(x) = \phi(y)$, then by algebraicity we have
		\begin{equation}
			x = \bigvee \phi(x) = \bigvee \phi(y) = y,
		\end{equation}
		so $\phi$ is injective. For any ideal $I \in \ideal(K(L))$, let $x = \bigvee I$ (where the join is taken in $L$). Therefore $\phi(x) = I$ by the same reasoning as in the join preservation proof and hence $\phi$ is surjective.
		
		For uniqueness, suppose $\ideal(P)\cong L$, we want to show that $P=K(L)$. Consider the inclusion map of posets:
		\begin{equation}
			\iota:P\rightarrow \ideal (P), \quad \iota(p)=\downarrow p.
		\end{equation}
		We only need to show that the principle ideals are exactly the compact elements of $\ideal(P)$. Trivially, the principle ideals are compact. Suppose $I\in \ideal(P)$ is compact. Since $I=\bigvee\{\downarrow p\ |\ p\in I\}$, there exists $p_1,\cdots ,p_n\in I$ such that $I=\downarrow p_1\vee\cdots \vee \downarrow p_n$. By the directedness of $I$, there exists $p\in I$ such that $p_i\leq p$ for $1\leq i\leq n$. Then we conclude that $I=\downarrow p$ is a principle ideal.
		
		This completes the proof.
	\end{proof}
	
	\begin{remark}\label{rmk:realization of extension ideal}
		Keep the notations above and let $\bot$ be the minimal element of $L$. Note that the map
		\begin{equation}
			\ideal (K(L))\rightarrow \ideal_0(K(L)\backslash\{\bot\}), \quad I\mapsto I\backslash\{\bot\}
		\end{equation}
		is an isomorphism of complete lattices. Therefore, we have $L\cong \ideal_0(K(L)\backslash \{\bot\})$, and the poset $K(L)\backslash \{\bot\}$ is uniquely determined up to isomorphism.
	\end{remark}
	
	\section{Stability approach to abelian categories}\label{sec:correspondence}
	Given an abelian category $\cA$, let $\tors \cA$ denote the set of torsion classes in $\cA$. It is shown in \cite{DIRRT} that $\tors \cA$ forms a complete lattice. For a nonempty subset $S \subseteq \tors \cA$, the meet $\bigwedge S$ is given by the intersection of the torsion classes in $S$, while the join $\bigvee S$ is the smallest torsion class containing all members of $S$.
	
	The authors in \cite{DK} considered maximal chains of torsion classes in $\tors \cA$. A \textbf{maximal green sequence} (possibly infinite) is a maximal chain of torsion classes with respect to this order. A chain of torsion classes is a totally ordered subposet of $\tors \cA$, and chains are ordered by inclusion. In this section we will discuss two important notions related to maximal green sequences: stability data and complete backward $\hom$-orthogonal sequences, and investigate their relationships in the infinite length setting.

	\subsection{Correspondence between the finest stability data and the maximal green sequence}
	
	Firstly, we give an equivalent condition to characterize the maximal green sequences.
	
	\begin{proposition}\label{prop:equivalence condition for mgs}
		A subset $\cS=\left\{\cT_{\alpha}\right\}_{\alpha\in J}$ of ${\rm tors}\cA$ is a maximal green sequence if and only if $\cS$ is a complete totally ordered set with respect to inclusion and satisfies:
		\begin{enumerate}
			\item $\bigwedge \{\cT_{\alpha}\mid \alpha\in J\}=\left\{0\right\}$ and $\bigvee \{\cT_{\alpha}\mid \alpha\in J\}=\cA$;
			\item If $\cT_{\alpha_0}$ is completely join-irreducible, then $ \cT_{\alpha_0}$ covers $\bigvee \left\{\cT_{\alpha}\mid \alpha<{\alpha_0}\right\}$ in ${\rm tors}\cA$;
		\end{enumerate}
		
		The notations $\bigvee$ and $\bigwedge$ above are the join and meet in $\cS$ respectively.
	\end{proposition}
	
	\begin{proof}
		The ``only if'' part is obvious.
		
		For the ``if'' part, suppose that $\cS=\left\{\cT_\alpha\right\}_{\alpha\in J}$ is a complete totally ordered set and satisfies (1) and (2). If $\cS'=\cS\amalg \left\{\cT_0\right\}$ is still a chain of torsion class, then we have the following relation:
		\begin{equation}
			\cT_{\alpha_0}=\bigvee\left\{\cT_{\alpha}\ |\ \cT_{\alpha}\subsetneq \cT_0,\ \alpha\in J\right\}\subsetneq\cT_0\subsetneq\bigwedge \left\{\cT_{\beta}\ |\ \cT_{\beta}\supsetneq \cT_0,\ \beta\in J\right\}=\cT_{\beta_0}.
		\end{equation}
		By definition, $\cT_{\alpha_0},\cT_{\beta_0}\in \cS$ and $\cT_{\beta_0}$ is completely join-irreducible. However, $\cT_{\beta_0}$ does not cover $\cT_{\alpha_0}$ which is a contradiction and we finish the proof.
	\end{proof}
	
	\begin{remark}
		\begin{enumerate}
			\item The condition (2) above can be replaced by the dual one, i.e. if $\cT_{\beta_0}$ is completely meet-irreducible, then $\bigwedge \left\{\cT_{\beta}\mid \beta>\beta_0\right\}$ covers $\cT_{\beta_0}$ in ${\rm tors}\cA$. Proposition \ref{prop:equivalence condition for mgs} can be viewed as an equivalent definition of maximal green sequence.
			\item If the index set $J$ of the maximal green sequence is finite, then our definition corresponds with the classical maximal green sequence since any finite totally ordered set is complete and any non-trivial $\cT_{\alpha}$ is both completely meet-irreducible and completely join-irreducible.
		\end{enumerate}
	\end{remark}
	
	Note that if $\cS=\left\{\cT_{\alpha}\right\}_{\alpha\in J}$ is a maximal green sequence, then the index set $J$ is also a complete totally ordered set whose order is given by that of $\cS$.
	
	Let $\left(\Phi,\left\{\Pi_{\varphi}\right\}_{\varphi \in \Phi}\right)$ be a stability data where $\Phi$ is a totally ordered set. Suppose that $F$ is a filter of $\Phi$. Define $F^c=\Phi\backslash F$. Obviously, $F^c$ is an ideal of $\Phi$. The following two simple lemmas have been established elsewhere in similar forms (see \cite[Proposition 2.8]{CLR} and \cite[Proposition 2.10]{Tr}). For convenience and consistency of notation, we restate them here and provide brief proofs.
	
	\begin{lemma}
		Suppose $\left(\Phi,\left\{\Pi_{\varphi}\right\}_{\varphi \in \Phi}\right)$ is a stability data on an abelian category $\cA$. Then any filter $F\subset \Phi$ can naturally induce a torsion pair $(\cT_{F}, \cF_{F})$ where 
		\begin{equation}
			\cT_{F}=\filter\left\{\Pi_{\varphi}\ |\ \varphi \in F\right\},\quad \cF_{F}=\filter\left\{\Pi_{\varphi} \ | \ \varphi\in F^c\right\}.
		\end{equation}
	\end{lemma}
	
	\begin{proof}
		By definition, we have $\hom_{\cA}(\cT_F, \cF_{F})=0$. Take any object $X\in \cA $. Consider the HN-filtration of $X$ with respect to $\left(\Phi, \left\{\Pi_{\varphi}\right\}_{\varphi\in \Phi}\right)$:
		\begin{equation}
			0=X_0\hookrightarrow X_1\hookrightarrow\cdots \hookrightarrow X_{n-1}\hookrightarrow X_n=X.
		\end{equation}
		There exists a unique $1\leq k\leq n-1$ such that $\varphi_k\in F$ and $\varphi_{k+1}\notin F$ since $F$ is a filter (The case that $\varphi_n\in F$ or $\varphi_1\notin F$ is trivial). Therefore, $X_{k}\in \cT_{F}$, $X/X_{k}\in \cF_{F}$ and hence the canonical sequence of $X$ is $0\to X_{k}\to X\to X/X_{k}\to 0$. We conclude that $(\cT_{F}, \cF_{F})$ is a torsion pair.
	\end{proof}
	
	For convenience, we define $\cT_{\emptyset}=\cF_{\emptyset}=\left\{0\right\}$. $\filter_0(\Phi)$ is a complete totally ordered set with respect to inclusion in which join is given by union and meet is given by intersection.
	
	Therefore, a stability data can induce a set of torsion classes:
	\begin{equation}\label{eq:torsion class induced by stability data}
		\cS_{\Phi}=\left\{\cT_{F}\ |\ F\in \filter_0(\Phi)\right\}.
	\end{equation}
	Note that $\cS_{\Phi}$ has a linear order with respect to inclusion. By our construction, for any $F_1,\ F_2\in \filter(\Phi)$, $F_1\subsetneq F_2$ implies $\cT_{F_1}\subsetneq \cT_{F_2}$ and $\cF_{F_1}\supsetneq \cF_{F_2}$ since each $\Pi_{\varphi}$ is non-zero.
	
	\begin{lemma}\label{lem:torsion pair induced by stability data}
		Suppose $(\Phi , \left\{\Pi_{\varphi}\right\}_{\varphi\in \Phi})$ is a stability data on an abelian length category $\cA$. $\cS_{\Phi}$ is the set of torsion classes as defined in \eqref{eq:torsion class induced by stability data}. Then for any $\cJ\subset \filter_0(\Phi)$, 
		\begin{equation}
			\left(\bigcup_{F\in \cJ} \cT_{F}, \bigcap_{F\in \cJ}\cF_{F}\right)
		\end{equation}
		and
		\begin{equation}
			\left(\bigcap_{F\in \cJ} \cT_{F}, \bigcup_{F\in \cJ}\cF_{F}\right)
		\end{equation}
		are torsion pairs. Furthermore, $\bigcup_{F\in \cJ} \cT_{F}$ and $\bigcap_{F\in \cJ} \cT_{F}$ are in $\cS_{\Phi}$.
	\end{lemma}
	
	\begin{proof}
		We only prove the first one and the other one is similar. 
		
		Firstly, we show that $\bigcup_{F\in \cJ} \cT_{F}$ is a torsion class. Let $\bar{F}:= \bigcup_{F\in \cJ} F\in \filter_0(\Phi)$. We only need to show that $\bigcup_{F\in \cJ} \cT_{F}=\cT_{\bar{F}}$. By our construction, we have $\bigcup_{F\in \cJ} \cT_{F}\subset\cT_{\bar{F}}$. For $X\in \cT_{\bar{F}}$, we have $\phi^-(X)\in \bar{F}$. Then $\phi^-(X)\in F_0$ for some $F_0\in \cJ$ and hence $X\in \cT_{F_0}$ since $F_0$ is a filter. 
		
		By definition, $(\bigcup_{F\in \cJ}\cT_{F}, \cF_{\bar{F}})$ is a torsion pair. We show $\bigcap_{F\in \cJ} \cF_F=\cF_{\bar{F}}$ by double inclusion. $\cF_{\bar{F}}\subset \bigcap_{F\in \cJ} \cF_{F}$ is obvious. For $Y\in \bigcap_{F\in \cJ}\cF_{F}$ and $X\in \cT_{\bar{F}}=\bigcup_{F\in \cJ} \cT_{F}$,  there exists $F_0\in \cJ$ such that $X\in \cT_{F_0}$. Therefore, $\hom_{\cA}(X, Y)=0$ and $\bigcap_{F\in \cJ} \cF_{F}\subset \cF_{\bar{F}}$. Then we finish the proof.
	\end{proof}
	
	The following theorem gives the relation between maximal green sequence and finest stability data in an abelian length category.
	
	\begin{theorem}\label{thm:mgs and finest stability data}
		Let $\left(\Phi , \left\{\Pi_{\varphi}\right\}_{\varphi\in \Phi}\right)$ be a stability data on an abelian length category $\cA$. $\cS_{\Phi}$ is as defined in \eqref{eq:torsion class induced by stability data}. Then $\cS_{\Phi}$ is a maximal green sequence if and only if $(\Phi , \left\{\Pi_{\varphi}\right\}_{\varphi\in \Phi})$ is a finest stability data.
	\end{theorem}
	
	\begin{proof}
		For the ``if'' part, we firstly show that $\cS_{\Phi}=\left\{\cT_{F}\ \middle|\ F\in \filter_0(\Phi)\right\}$ is a complete totally ordered set. By definition, $\cS_{\Phi}$ is a totally ordered set with respect to inclusion. Suppose $\cJ$ is a subset of $\filter_0(\Phi)$. Then by the proof in Lemma \ref{lem:torsion pair induced by stability data}, we know that $\bigcup_{F\in \cJ}\cT_{F}\in \cS_{\Phi}$. By our construction, the join of $\left\{\cT_{F}\right\}_{F\in \cJ}$ is given by $\bigvee \{\cT_{F}\mid F\in\cJ\} := \bigcup_{F\in \cJ}\cT_{F}$. Dually, the meet of $\left\{\cT_{F}\right\}_{F\in \cJ}$ is given by $\bigwedge \{\cT_{F}\mid F\in\cJ\} := \bigcap_{F\in \cJ}\cT_{F}\in \cS_{\Phi}$.
		
		Note that $\cT_{\emptyset}=\left\{0\right\}$ and $\cT_{\Phi}=\cA$ and hence (1) in Proposition \ref{prop:equivalence condition for mgs} follows. To prove (2) in Proposition \ref{prop:equivalence condition for mgs}, suppose that $\cT_{F_0}\in \cS_{\Phi}$ is completely join-irreducible. Consider $\cT_{\bar{F}}=\bigvee \{\cT_{F}\mid F< F_0\} = \bigcup_{F< F_0} \cT_{F}$ where $\bar{F}:= \bigcup_{F< F_0} F\in \filter_0(\Phi)$. Assume that there exists a torsion class $\cT \in{\rm tors}\cA$ such that $\cT_{\bar{F}}\subsetneq \cT\subsetneq \cT_{F_0}$. By our construction, there exists a unique $\varphi_0\in F_0\backslash\bar{F}$ such that $\Pi_{\varphi_0}\ne \left\{0\right\}$. Indeed, if there exist $\varphi_0, \varphi_0'\in F_0\backslash\bar{F}$ and $\Pi_{\varphi_{0}}\ne \left\{0\right\}\ne \Pi_{\varphi_{0}'}$, suppose $\varphi_{0}<\varphi_{0}'$, then the filter $\uparrow\varphi_{0}'$ is a proper subset of $F_0$ and hence $\Pi_{\varphi_{0}'}\subset \cT_{\uparrow\varphi_{0}'}\subset \cT_{\bar{F}}$. Therefore, $\varphi_{0}'\in \bar{F}$. Therefore there exists a nonzero object $X\in \cT\backslash\cT_{\bar{F}}$. Denote by $A_n$ the minimal semistable quotient object of $X$. Then by definition, $\phi(A_n)\in {F_0}\backslash \bar{F}$. Since $A_n$ is nonzero, we have $\Pi_{\phi(A_n)}\ne 0$. By the uniqueness of $\varphi_0$, we have $\phi(A_n)=\varphi_0$. In addition, we have $A_n\in \cT$ since $\cT$ is a torsion class and $A_n$ is a quotient of $X$. Let $\cF=\cT^{\perp}$, then we have $\cF_{F_0}\subsetneq \cF\subsetneq \cF_{\bar{F}}$. For the same reason, there exists a nonzero object $B_m\in \cF$ such that $\phi(B_m)=\varphi_0$, i.e. $B_m\in \Pi_{\varphi_0}$. However, since $(\cT, \cF)$ is a torsion pair, we have $\hom_{\cA}(A_n, B_m)=0$. By Theorem \ref{thm:finest stability data}, $(\Phi,\left\{\Pi_{\varphi}\right\}_{\varphi\in \Phi})$ is not a finest stability data, a contradiction.
		
		For the ``only if'' part, suppose $(\Phi,\left\{\Pi_{\varphi}\right\}_{\varphi\in \Phi})$ is not a finest stability data, then by Theorem \ref{thm:finest stability data}, there exists $\varphi_{0}\in \Phi$ and nonzero objects $X, Y\in \Pi_{\varphi_{0}}$ such that $\hom_{\cA}(X,Y)=0$. Consider the torsion class $\cT_{\uparrow\varphi_{0}}$. Define $\bar{F}:=\bigcup_{F<\uparrow\varphi_{0}} F=\uparrow\varphi_{0}\backslash\left\{\varphi_0\right\}$. Then $\bigvee \{\cT_{F}\mid F< \uparrow\varphi_{0}\} = \bigcup_{F<\uparrow\varphi_{0}}\cT_{F}=\cT_{\bar{F}}\subsetneq \cT_{\uparrow\varphi_{0}}$ which means that $\cT_{\uparrow\varphi}$ is completely join-irreducible since $\Pi_{\varphi}\ne \left\{0\right\}$. We claim that $\cT_{\uparrow\varphi_{0}}$ does not cover $\cT_{\bar{F}}$. To see this, consider the minimal torsion class which contains $\cT_{\bar{F}}$ and $X$, namely, $\cT:=\filt(\fac(\cT_{\bar{F}}\cup \{X\}))$. 
		On one hand, it is obvious that $\cT_{\bar{F}}\subsetneq \cT$ since $\hom_{\cA}(\cT_{\bar{F}}, X)=0$. On the other hand, we have $\hom_{\cA}(X, Y)=0$ and $\hom_{\cA}(\cT_{F}, Y)=0$ which implies that $\cT \subset \prescript{\perp}{}{Y}$. It follows that $Y\notin \cT$ and hence $\cT\subsetneq \cT_{\uparrow\varphi_{0}}$. Therefore, $\cS_{\Phi}$ is not a maximal green sequence and we finish the proof.
	\end{proof}
	
	Theorem \ref{thm:mgs and finest stability data} can be viewed as a generalization of the following proposition, which establishes the relationship between finite finest stability data and (classical) maximal green sequences in an abelian category. The main difficulty in our proof lies in the fact that, when dealing with maximal green sequences of infinite length, the usual notion of “adjacent” torsion classes is no longer available. Consequently, one must instead work with the concept of completely join‑irreducible torsion classes to capture the necessary combinatorial structure.
	
	\begin{proposition}\cite[Proposition 3.5]{CLR}
		Let $\left(\Phi, \left\{\Pi_{\varphi}\right\}_{\varphi\in \Phi}\right)$ be a finite stability data on an abelian category $\cA$, which induces a chain of torsion classes. Then it is a (classical) maximal green sequence if and only if $\left(\Phi, \left\{\Pi_{\varphi}\right\}_{\varphi\in \Phi}\right)$ is finest.
	\end{proposition}
	
	Now given an abelian length category $\cA$, by Theorem \ref{thm:mgs and finest stability data}, we obtain a well-defined map $\eta$ between these two set:
	\begin{equation}\label{eq:map from finest stability data to mgs}
		\eta:\left\{\text{finest stability data on }\cA\right\}\rightarrow\left\{\text{maximal green sequences in }\cA\right\}
	\end{equation}
	
	We aim to prove that the map $\eta$ is a bijection by constructing stability data on $\cA$ from a chain of torsion classes. To this end, we first establish the following framework.
	
	Let $\cS = \left\{\cT_\alpha\right\}_{\alpha \in J}$ be a maximal green sequence in $\cA$. The index set $J$ is endowed with a partial order induced by inclusion within $\cS$; specifically, $\alpha < \beta$ if and only if $\cT_\alpha \subsetneq \cT_\beta$.
	
	For a general lattice, there is no direct relationship between compact elements and completely join-irreducible elements. However, they have the following relation in a complete totally ordered set.
	
	\begin{lemma}\label{lem:compact element}
		Suppose $J$ is a complete totally ordered set with the minimal element $\bot$ and the maximal element $\top$, then the following holds:
		\begin{enumerate}
			\item $x\in J$ is compact if and only if $x=\bot$ or $x$ is completely join-irreducible.
			\item $x\in J$ is co-compact if and only if $x=\top$ or $x$ is completely meet-irreducible.
		\end{enumerate}
	\end{lemma}
	
	\begin{proof}
		We only prove (1). 
		
		For the ``if'' part, obviously $\bot$ is compact. Now suppose $x$ is completely join-irreducible and $x \leq \bigvee I$ where $I\subset J$. Indeed, there exists $y\in I$ such that $x\leq y$.  Otherwise, for any $y\in I$, we have $y<x$ ($J$ is linearly ordered). It is impossible since $x$ is completely join-irreducible.
		
		For the ``only if'' part, suppose $x\in J$ is compact and $x\ne \bot$. Suppose $x=\bigvee I$ where $I\subset J$ and $I\ne \emptyset$. There exists finitely many elements, say, $\{y_1, \cdots y_n\}\subset I$ such that $x\leq\bigvee\{y_1, \cdots y_n\}$. Since $J$ is linearly ordered, there exists a maximal element, say, $y_1$ in $\{y_1, \cdots y_n\}$. Thus we have $y_1\leq x\leq y_1$ which yields that $x$ is completely join-irreducible.
	\end{proof}
	
	It is already showed that the set of torsion classes ${\rm tors} \cA$ of an abelian length category $\cA$ is a complete algebraic lattice\cite{DIRRT}. However, a complete sub-lattice of ${\rm tors} \cA$ is not necessarily again algebraic. Now we show that any maximal green sequence is algebraic.
	
	\begin{proposition}
		Suppose $\{\cT_{\alpha}\}_{\alpha\in J}$ is a maximal green sequence in an abelian length category $\cA$. Then $\{\cT_{\alpha}\}_{\alpha\in J}$ is algebraic as a complete lattice.
	\end{proposition}
	
	\begin{proof}
		The authors in \cite[Proposition A.2]{DK} have proved that each torsion class $\cT_{\alpha}$ (except $0$) equals the join of the completely join-irreducible elements contained in $\cT_{\alpha}$. By Lemma \ref{lem:compact element}, we know that $\cT_{\alpha}$ is algebraic.
	\end{proof}
	
	By taking a dual of Proposition \ref{prop:realization of ideal} and Remark \ref{rmk:realization of extension ideal}, we know that there exists a unique (up to isomorphism) totally ordered set $\Phi$ such that $\filter_0(\Phi)\cong J$. More precisely, combining Lemma \ref{lem:compact element}, $\Phi$ is isomorphic to the set of completely join-irreducible elements of $J$ with opposite order. 
	
	For any $\varphi\in \Phi$, define 
	\begin{equation}
		\Pi_{\varphi}=\cT_{\uparrow\varphi}\cap \cF_{\uparrow\varphi\backslash\{\varphi\}}.
	\end{equation}
	We claim that $\left(\Phi, \Pi_{\varphi}\right)$ is the preimage of $J$ under the map $\eta$ defined in \eqref{eq:map from finest stability data to mgs}.
	
	\begin{proposition}\label{prop:mgs to finset stability data}
		Suppose $\cS=\{\cT_{\alpha}\}_{\alpha\in J}$ is a maximal green sequence in an abelian length category $\cA$, $\left(\Phi, \Pi_{\varphi}\right)$ is as defined above such that $\filter_0(\Phi)\cong J$. Then it is a finest stability data on $\cA$.
	\end{proposition}
	
	\begin{proof}
		By the ``only if'' part of Theorem \ref{thm:mgs and finest stability data}, we only need to show that $\left(\Phi, \{\Pi_{\varphi}\}_{\varphi\in \Phi}\right)$ is a stability data. By our construction, $\Pi_{\varphi}\ne \{0\}$ for any $\varphi\in \Phi$.
		
		(1) Suppose that $\varphi_1 <_\Phi \varphi_2$ and then $\uparrow\varphi_2 <_J \uparrow\varphi_1$. Therefore $\uparrow\varphi_2 \leq_J \uparrow\varphi_1\backslash\{\varphi_1\}$ and $\Pi_{\varphi_2}\subset \cT_{{\uparrow\varphi_1}\backslash\varphi_1}$. It follows that $\hom_{\cA}(\Pi_{\varphi_2},\Pi_{\varphi_1})=0$.
		
		(2) Now consider the chain of torsion classes $\{\cT_{\alpha}\}_{\alpha\in J}$ and an object $X\in \cA$. Let $X_n:=X$ and define $\cS_n:=\left\{\cT_{\alpha}\in J\mid X_n\in \cT_{\alpha}\right\}$. The set $\cS_n$ clearly has a minimal torsion class (by taking meet), which we denote by $\cT_{\alpha_n}$. Similarly, the complement $\cS_n^c$ has a maximal torsion class, denoted by $\cT_{\beta_n}$. 
		
		It follows that $\cT_{\alpha_n}$ covers $\cT_{\beta_n}$, making $\cT_{\alpha_n}$ completely join-irreducible in $\cS$. Therefore, there exists some $\varphi_n\in \Phi$ such that $\alpha_n=\uparrow \varphi_n$, and consequently $\beta_n=\uparrow\varphi_n\backslash\{\varphi_n\}$. 
		
		Now consider the canonical sequence of $X_n$ with respect to the torsion pair $(\cT_{\uparrow \varphi_n\backslash \{\varphi_n\}},\cF_{\uparrow \varphi_n\backslash \{\varphi_n\}})$:
		\begin{equation}
			0\rightarrow X_{n-1}\rightarrow X_n\rightarrow X_n/X_{n-1}\rightarrow 0.
		\end{equation}
		We then have $X_n/X_{n-1}\in \Pi_{\varphi_n}$. 
		
		If $X_{n-1}=0$, we set $n=1$ and obtain the HN-filtration $0\subset X_1=X$ for $X$.
		
		Otherwise, we replace $X_n$ by $X_{n-1}$ and repeat the above procedure. Since $X$ has finite length, this process terminates after finitely many steps. Here, $n$ denotes the total number of steps performed. Then we obtain a filtration:
		\begin{equation}
			0=X_0\subset X_1\subset\cdots \subset X_{n-1}\subset X_n=X.
		\end{equation}
		
		By construction, we have $\cT_{\uparrow\varphi_n\backslash \{\varphi_n\}}\in \cS_{n-1}$, and $\cT_{\uparrow\varphi_{n-1}}$ is the minimal torsion class in $\cS_{n-1}$. This implies that $\cT_{\uparrow\varphi_{n-1}}\subsetneq \cT_{\uparrow\varphi_{n}}$, and hence $\varphi_{n}<\varphi_{n-1}$. By induction, we conclude that $\varphi_{i}<\varphi_{i-1}$ for all $i$.
		
		Therefore, $\left(\Phi, \{\Pi_{\varphi}\}_{\varphi\in \Phi}\right)$ is a stability data. By Theorem \ref{thm:mgs and finest stability data}, $\left(\Phi, \{\Pi_{\varphi}\}_{\varphi\in \Phi}\right)$ is finest. Then we finish the proof.
	\end{proof}
	
	As a corollary, we know that the map $\eta$ defined in \eqref{eq:map from finest stability data to mgs} is actually a bijection.
	
	The authors in \cite{CLR} proved that in the tube category, any stability data can be refined to a finest one. Now we show that it holds in any abelian length category by lattice-theoretic methods.
	
	\begin{corollary}
		Let $\cA$ be an abelian length category. Then every stability data on $\cA$ admits a refinement to a finest stability data.
	\end{corollary}
	
	\begin{proof}
		Let $\left(\Phi, \{\Pi_{\varphi}\}_{\varphi\in \Phi}\right)$ be a stability data on $\cA$, and denote by $\cS_{\Phi} := \{\cT_F \mid F \in \filter_0(\Phi)\}$ the chain of torsion classes induced by $\Phi$. Consider the collection
		\begin{equation}
			\cC := \{\cS \mid \cS \text{ is a chain in } \tors\cA \text{ containing } \cS_{\Phi}\}.
		\end{equation}
		Partially ordered by inclusion, $\cC$ is non-empty, and every chain in $\cC$ has an upper bound (given by the union). By Zorn's Lemma, $\cC$ contains a maximal element $\cS_0 = \{\cT_{\alpha}\}_{\alpha \in J}$, which is a maximal green sequence. Let $\left(\Psi, \{P_\psi\}_{\psi \in \Psi}\right)$ be the finest stability data corresponding to $\cS_0$ under the bijection $\eta$ (as in Proposition~\ref{prop:mgs to finset stability data}). Then $\filter_0(\Phi)$ is a sub-lattice of $\filter_0(\Psi)$.
		
		We now show that $\left(\Psi, \{P_\psi\}_{\psi \in \Psi}\right)$ refines $\left(\Phi, \{\Pi_{\varphi}\}_{\varphi \in \Phi}\right)$. Define a map $r \colon \Psi \to \Phi$ as follows: for each $\psi \in \Psi$, choose a nonzero object $X \in P_\psi = \cT_{\uparrow \psi} \cap \cF_{\uparrow \psi \setminus \{\psi\}}$. Consider the Harder--Narasimhan filtration of $X$ with respect to $\left(\Phi, \{\Pi_{\varphi}\}_{\varphi \in \Phi}\right)$ (as constructed in the proof of Proposition~\ref{prop:mgs to finset stability data}). Let $\cT_{\uparrow \varphi}$ be the minimal torsion class in $\{\cT_F \in \cS_\Phi \mid X \in \cT_F\}$, which is completely join-irreducible in $\filter_0(\Phi)$ and satisfies $\varphi = \phi^-(X)$ under $\left(\Phi, \{\Pi_{\varphi}\}_{\varphi \in \Phi}\right)$. Then $\cT_{\uparrow \psi} \subseteq \cT_{\uparrow \varphi}$, so $\uparrow \psi \leq \uparrow \varphi$ in $\filter_0(\Psi)$. Suppose that $\cT_{\uparrow \varphi}$ covers $\cT_{F_0}$ in $\cS_\Phi$. Since $X \in \cT_{\uparrow \psi}$ but $X \notin \cT_{F_0}$, we have $F_0 < \uparrow \psi$, and hence $F_0 \leq \uparrow \psi \setminus \{\psi\}$. This implies $\cF_{\uparrow \psi \setminus \{\psi\}} \subseteq \cF_{F_0}$, and thus $P_\psi \subseteq \Pi_{\varphi}=\cT _{\uparrow \varphi}\cap \cF_{F_0}$. Define $r(\psi) = \varphi$. In other words, each nonzero object in $P_{\psi}$ is semistable and has phase $\varphi$ with respect to $\left(\Phi, \{\Pi_{\varphi}\}_{\varphi \in \Phi}\right)$. Suppose $\psi_1<\psi_2$ and $X_i\in P_{\psi_i}$ where $i=1,2$. By our construction, $\cT_{\uparrow r(\psi_1)}$ is the minimal torsion class in $\{\cT_{F}\in \cS_{\Phi}\mid X_1\in \cT_F\}$. Note that $X_1\in \cT_{F}$ implies $X_2\in \cT_{F}$, we have $X_2\in \cT_{\uparrow r(\psi_1)}$ and hence $\cT_{\uparrow  r(\psi_2)}\subset \cT_{\uparrow r(\psi_1)}$ by definition. Therefore $r(\psi_1)\leq r(\psi_2)$, i.e. $r$ is order-preserving.
		
		Conversely, for any $\varphi \in \Phi$ and any nonzero object $X \in \Pi_{\varphi} = \cT_{\uparrow\varphi} \cap \cF_{F_0}$ (where $\uparrow\varphi$ covers $F_0$ in $\filter_0(\Phi)$), consider the Harder--Narasimhan filtration of $X$ with respect to $\left(\Psi, \{P_\psi\}_{\psi \in \Psi}\right)$:
		\begin{equation}
			0 = X_0 \subset X_1 \subset \cdots \subset X_n = X,
		\end{equation}
		with factors $A_i = X_i/X_{i-1} \in P_{\psi_i}$ and phases $\psi_1 > \psi_2 > \cdots > \psi_n$. We claim that each $A_i$ has phase $\varphi$ with respect to $\left(\Phi, \{\Pi_{\varphi}\}_{\varphi \in \Phi}\right)$.
		
		First, note that $A_n \in \cT_{\uparrow\varphi}$ because $X \in \cT_{\uparrow\varphi}$. If $A_n$ were in $\cT_{F_0}$, then since $\psi_i > \psi_n$ for all $i < n$, we would have $A_i \in \cT_{F_0}$ for all $i$, implying $X \in \cT_{F_0}$, a contradiction. Hence, $A_n$ has phase $\varphi$. Dually, by considering the torsion-free parts, we conclude that $A_1$ (and hence every $A_i$) has phase $\varphi$ since $r$ is order-preserving. Therefore, $r(\psi_i) = \varphi$ for all $i$.
		
		We conclude that $\left(\Psi, \{P_\psi\}_{\psi \in \Psi}\right)$ refines $\left(\Phi, \{\Pi_{\varphi}\}_{\varphi \in \Phi}\right)$ and we finish the proof.
	\end{proof}
	
	\subsection{Correspondence with complete backward \texorpdfstring{$\hom$}{hom}-orthogonal sequences}
	
	As discussed in \cite{BCZ} and \cite{DIRRT}, bricks play a central role in the representation theory, especially in the study of torsion classes. And the relation between (classical) maximal green sequences and certain chains of bricks were discussed in \cite{LL}. In this section, we will investigate the relation between bricks and the stability data in an abelian category.
	
	\begin{definition}
		Let $\cA$ be an abelian category. An object $B$ of $\mathcal{A}$ is called a \textbf{brick} if \  ${\mathrm End}_{\cA} B$ is a division ring.
	\end{definition}
	
	For each semistable subcategory $\Pi_{\varphi}$, we have the following lemma which can be compared with \cite[Proposition 3.7]{Tr}, which says that if $B$ is both relatively simple and relatively cosimple, then $B$ is a brick.
	
	\begin{lemma}\label{lem:finset stability data to brick}
		Suppose that $\mathcal{A}$ is an abelian category, $\left(\Phi,\left\{\Pi_{\varphi}\right\}_{\varphi \in \Phi}\right)$ is a stability data on $\mathcal{A}$.
		\begin{enumerate}
			\item If $B\in \Pi_{\varphi}$ has no nontrivial sub-object in $\Pi_{\varphi}$, then $B$ is a brick;
			\item If $B\in \Pi_{\varphi}$ has no nontrivial quotient object in $\Pi_{\varphi}$, then $B$ is a brick;
		\end{enumerate}
	\end{lemma}
	
	\begin{proof}
		We only prove (1). Suppose that $f\in {\mathrm End}_{\cA} B$ is not an isomorphism. Consider the usual kernel-image and image-cokernel exact sequences in which $I$ is nontrivial:
		\begin{align}
			0 \rightarrow K \rightarrow B \rightarrow I \rightarrow 0 \\
			0 \rightarrow I \rightarrow B \rightarrow C \rightarrow 0
		\end{align}
		By Lemma \ref{lem:phase in exact sequence}, we have $\phi^-(I)\ge \phi^-(B)$ and $\phi^+(B)\ge \phi^+(I)$. Note that $B$ is a semistable object in $\Pi_{\varphi_0}$, the following inequalities hold:
		\begin{equation}
			\phi^+(I)\ge \phi^-(I)\ge \phi^-(B) =\varphi_0= \phi ^+(B)\ge \phi^+(I)
		\end{equation}
		Then we have that  $\phi^-(I)=\phi^+(I)$. Thus $I$ is a semistable object of phase $\varphi_0$ and lies in $\Pi_{\varphi_0}$. This contradicts to the simplicity of $B$.
	\end{proof}
	
	\begin{proposition}\label{prop:finset stability data to chain of brick}
		Suppose that $\cA$ is an abelian length category. $\left(\Phi,\left\{\Pi_{\varphi}\right\}_{\varphi \in \Phi}\right)$ is a stability data on $\cA$. Then $\left(\Phi,\left\{\Pi_{\varphi}\right\}_{\varphi \in \Phi}\right)$ is finest if and only if for any $\varphi\in \Phi$, $\Pi_{\varphi}$ is filtered by exactly one brick.
	\end{proposition}
	
	\begin{proof}
		For the ``if'' part, suppose $\varphi\in\Phi$ and $\Pi_{\varphi}=\filt\left\{B_{\varphi}\right\}$. Then for any two nonzero objects $X,Y\in \Phi_{\varphi}$, the composition of morphisms $X\twoheadrightarrow B_{\varphi}\hookrightarrow Y$ is nonzero. By Theorem \ref{thm:finest stability data}, $\left(\Phi,\left\{\Pi_{\varphi}\right\}_{\varphi \in \Phi}\right)$ is finest.
		
		For the ``only if'' part, fix a semistable subcategory $\Pi_{\varphi}$ of $\cA$. Since $\cA$ is an abelian length category, for each nonzero $\Pi_{\varphi}$, there exists an object $B\in \Pi_{\varphi}$ such that $B$ has minimal length. Thus $B$ has no nontrivial subobject in $\Pi_{\varphi}$. By Lemma \ref{lem:finset stability data to brick}, $B$ is a brick. Since $\left(\Phi,\left\{\Pi_{\varphi}\right\}_{\varphi \in \Phi}\right)$ is finest, for any $X\in \Pi_{\varphi}$, there exists nonzero morphism $f: B\to X$ and $g: X\to B$ by Theorem \ref{thm:finest stability data}. 
		
		We claim that $f$ is monic and $g$ is epic. Consider the canonical decomposition of $f: B\twoheadrightarrow \im f\hookrightarrow X$. Since $\im f$ is a quotient object of $B$, we have $\phi^-(\im f)\ge \phi(B)=\varphi$. Similarly, $\phi^+(\im f)\leq \phi(X)=\varphi$. Combine these two inequalities, we have $\phi^+(\im f)=\phi^-(\im f)=\varphi$ and thus $\im f\in \Pi_{\varphi}$. By the choice of $B$, we have $B\cong \im f$, otherwise the length of $\im f$ is strictly less than that of $B$ and this is a contradiction. Therefore, $f$ is monic and dually, $g$ is epic. Then we have proven the claim.
		
		Next we prove that $B$ is the unique brick in $\Pi_{\varphi}$. Suppose that $B'$ is another brick, then there exists a monomorphism $f:B\to B'$ and an epimorphism $g:B'\to B$. Thus $g\circ f\in {\rm End}_{\cA} B'$ is nonzero and therefore an isomorphism. Then $B'\cong B$ and the uniqueness follows.
		
		Finally, we show that $\Pi_{\varphi}=\filt\left\{B\right\}$. The direction ``$\supset$'' holds since $\Pi_{\varphi}$ is extension-closed. For any nonzero indecomposable object $X\in \Pi_{\varphi}$ and $X\not\cong B$, let $f,g$ be as defined above. Since $gf\in {\rm End}_{\cA}B$, $gf$ is either $0$ or an isomorphism. If $gf={\rm id}_{S}$, then $f$ is a section and $X\cong B\oplus C$ which is a contradiction. Thus $gf=0$. We have the following commutative diagram:
		\begin{equation}
			\begin{tikzcd}
				&& 0 & 0 \\
				0 & B & K & M & 0 \\
				0 & B & X & C & 0 \\
				&& B & B \\
				&& 0 & 0
				\arrow[from=1-3, to=2-3]
				\arrow[from=1-4, to=2-4]
				\arrow[from=2-1, to=2-2]
				\arrow[from=2-2, to=2-3]
				\arrow[equals, from=2-2, to=3-2]
				\arrow[from=2-3, to=2-4]
				\arrow["\iota", from=2-3, to=3-3]
				\arrow[from=2-4, to=2-5]
				\arrow[from=2-4, to=3-4]
				\arrow[from=3-1, to=3-2]
				\arrow["f", from=3-2, to=3-3]
				\arrow["\pi", from=3-3, to=3-4]
				\arrow["g", from=3-3, to=4-3]
				\arrow[from=3-4, to=3-5]
				\arrow[from=3-4, to=4-4]
				\arrow[equals, from=4-3, to=4-4]
				\arrow[from=4-3, to=5-3]
				\arrow[from=4-4, to=5-4]
			\end{tikzcd}
		\end{equation}
		Then by definition, we have $\varphi=\phi^-(X)\leq \phi^-(C)\leq \phi^-(B)=\varphi$ and thus $\phi^-(C)=\varphi$. Assume, if possible, that $\phi^+(C)>\varphi$, let $D$ be the minimal semistable quotient object of $C$ and $p:C\rightarrow D$ be the quotient morphism. Then $\ker p\ne 0$ and $\phi^-(\ker p)>\varphi$. Consider the pullback of $0\rightarrow B\xrightarrow{f}X\xrightarrow{\pi}C\rightarrow 0$ along the morphism $\ker p\hookrightarrow C$:
		\begin{equation}
			\begin{tikzcd}
				&& 0 & 0 \\
				0 & B & {X'} & {\ker p} & 0 \\
				0 & B & X & C & 0 \\
				&& D & D \\
				&& 0 & 0
				\arrow[from=1-3, to=2-3]
				\arrow[from=1-4, to=2-4]
				\arrow[from=2-1, to=2-2]
				\arrow[from=2-2, to=2-3]
				\arrow[equals, from=2-2, to=3-2]
				\arrow[from=2-3, to=2-4]
				\arrow[from=2-3, to=3-3]
				\arrow[from=2-4, to=2-5]
				\arrow[from=2-4, to=3-4]
				\arrow[from=3-1, to=3-2]
				\arrow["f", from=3-2, to=3-3]
				\arrow["\pi", from=3-3, to=3-4]
				\arrow[from=3-3, to=4-3]
				\arrow[from=3-4, to=3-5]
				\arrow["p", from=3-4, to=4-4]
				\arrow[equals, from=4-3, to=4-4]
				\arrow[from=4-3, to=5-3]
				\arrow[from=4-4, to=5-4]
			\end{tikzcd}
		\end{equation}
		Then again by Lemma \ref{lem:phase in exact sequence} (1), we have $\varphi=\phi^+(B)\leq \phi^+(X')\leq \phi^+(X)=\varphi$, and thus $\phi^+(X')=\varphi$. In addition, by Lemma \ref{lem:phase in exact sequence} (2), we have $\phi^-(X')\ge \min\left\{\phi^-(B),\phi^-(\ker p)\right\}=\varphi$ and thus $X'\in \Pi_{\varphi}$. By induction on the length of $X$ and use the hypothesis, we know that $X\in \filt\left\{B\right\}$.
	\end{proof}
	
	\begin{corollary}
		Suppose that $\cA$ is an abelian length category. $\left(\Phi,\left\{\Pi_{\varphi}\right\}_{\varphi \in \Phi}\right)$ is a finest stability data on $\cA$. Then $\Pi_{\varphi}$ is a wide subcategory of $\cA$ for any $\varphi\in \Phi$, i.e. it is closed under kernels, cokernels and extensions.
	\end{corollary}
	
	\begin{proof}
		By Proposition \ref{prop:finset stability data to chain of brick}, we know that for any $\varphi\in \Phi$, $\Pi_{\varphi}=\filt\left\{B_{\varphi}\right\}$ for some brick $B_{\varphi}\in \Pi_{\varphi}$. We've known from \cite[Theorem 2.5]{En} that $\filt\left\{B_{\varphi}\right\}$ is a wide subcategory of $\cA$.
	\end{proof}
	
	\begin{definition}
		Suppose $\cA$ is an abelian category, $\left\{B_\varphi\right\}_{\varphi\in \Phi}$ is a set of bricks in $\cA$ where $\Phi$ is a totally ordered set. We say $\left\{B_\varphi\right\}_{\varphi\in \Phi}$ is a \textbf{backward $\hom$-orthogonal sequence} if $\hom_{\cA}(B_{\varphi'},B_{\varphi})=0$ for any $\varphi<\varphi'$. It is called \textbf{complete} if it cannot be embedded into a longer backward $\hom$-orthogonal sequence. In this case, we write it as \textbf{CBHO sequence} for short.
	\end{definition}
	
	The following lemma can be viewed as a generalization of \cite[Lemma 2.5]{Ig1} which discuss on the finite case.
	
	\begin{lemma}
		Suppose $\left\{B_\varphi\right\}_{\varphi\in \Phi}$ is a CBHO sequence in an abelian category $\cA$ and $S\in \cA$ is a simple object. If $S\notin \left\{B_\varphi\right\}_{\varphi\in \Phi}$, then $\hom_{\cA}(B_{\varphi}, S)=0$ for all $\varphi\in \Phi$. In particular, any CBHO sequence must contain every simple object in $\cA$.
	\end{lemma}
	
	\begin{proof}
		Suppose $S \notin \{B_{\varphi}\}_{\varphi \in \Phi}$ and
		\begin{equation}
			P := \{\varphi \in \Phi \mid \hom_{\cA}(B_{\varphi}, S) \ne 0\}
		\end{equation}
		is nonempty. We claim that for any $\varphi \in {\downarrow P}$, we have $\hom_{\cA}(S, B_{\varphi}) = 0$.
		
		To prove this, first consider the case $\varphi \in P$. If $\hom_{\cA}(S, B_{\varphi}) \ne 0$, then the composition
		\begin{equation}
			B_{\varphi} \to S \to B_{\varphi}
		\end{equation}
		yields a nonzero radical endomorphism of $B_{\varphi}$, which contradicts the assumption that $B_{\varphi}$ is a brick.
		
		Now suppose $\varphi \in {\downarrow P} \setminus P$ (if nonempty) and $\hom_{\cA}(S, B_{\varphi}) \ne 0$. Since $\varphi \in {\downarrow P}$, there exists some $\varphi_0 \in P$ such that $\varphi < \varphi_0$. Then the composition
		\begin{equation}
			B_{\varphi_0} \to S \to B_{\varphi}
		\end{equation}
		is a nonzero morphism in $\hom_{\cA}(B_{\varphi_0}, B_{\varphi})$, contradicting the $\hom$-orthogonality condition.
		
		Hence, in both cases, we must have $\hom_{\cA}(S, B_{\varphi}) = 0$.
		
		In particular, suppose $\{B_{\varphi}\}_{\varphi \in \Phi}$ is a CBHO sequence. Consider the new index set 
		\begin{equation}
			\downarrow P \sqcup \{*\} \sqcup (\Phi \setminus \downarrow P)
		\end{equation}
		with $B_* = S$, where the ordering is extended by declaring $* > \varphi$ for all $\varphi \in \downarrow P$ and $* < \varphi$ for all $\varphi \in \Phi \setminus \downarrow P$. Then this forms a CBHO sequence strictly containing $\{B_{\varphi}\}_{\varphi \in \Phi}$, which contradicts the maximality assumption.
	\end{proof}
	
	By Proposition \ref{prop:finset stability data to chain of brick}, we obtain a well-defined map
	\begin{equation}
		\xi:\left\{\text{finest stability data on }\cA\right\}\to\left\{\text{CBHO sequence in }\cA\right\}.
	\end{equation}
	We show that $\xi$ is indeed a bijection.
	
	For a CBHO sequence $\{B_\varphi\}_{\varphi\in \Phi}$ in $\cA$, the authors in \cite{DK} proved that there is a bijection
	\begin{equation}
		\zeta:\left\{\text{CBHO sequence in }\cA\right\}\to\left\{\text{maximal green sequence of }\cA\right\}
	\end{equation}
	where the maximal green sequence is indexed by $\filter(\Phi)$ and for $F\in \filter(\Phi)$, $\cT_{F}$ is the minimal torsion class which contains $\{B_\varphi\}_{\varphi\in F}$. Notice that $\zeta\circ \xi=\eta$ and $\eta$ is also a bijection, we know that $\xi$ is also a bijection.
	
	Now we give the main theorem of this article.
	
	\begin{theorem}\label{thm:correspondence of three sets}
		Let $\cA$ be an abelian length category. Then there exist natural bijections between the following three sets:
		
		\begin{enumerate}
			\item The set of finest stability data on $\cA$;
			\item The set of maximal green sequences of $\cA$;
			\item The set of CBHO sequences in $\cA$.
		\end{enumerate}
		
		More explicitly, the bijections are given by the maps
		\[
		\xi: \text{(1)} \longrightarrow \text{(3)}, \qquad
		\zeta: \text{(3)} \longrightarrow \text{(2)}, \qquad
		\eta: \text{(1)} \longrightarrow \text{(2)},
		\]
		where $\eta = \zeta \circ \xi$. All three maps are bijective, and therefore any one of these structures determines the other two uniquely.
	\end{theorem}
	
	\begin{remark}
		By our construction of $\eta$ and the discussion above, the map $\zeta$ can be described more explicitly: for a CBHO sequence $\{B_\varphi\}_{\varphi\in \Phi}$ in $\cA$, the corresponding maximal green sequence $\zeta(\{B_\varphi\}_{\varphi\in \Phi})$ is indexed by $\filter(\Phi)$, with
		\begin{equation}
			\cT_F = \filt\{B_{\varphi}\mid \varphi\in F\}
		\end{equation}
		for each $F \in \filter(\Phi)$.
	\end{remark}
	
	In the next section, we will give an example for theorem \ref{thm:correspondence of three sets}.
	
	\subsection{Maximal green sequence of finite dimensional algebras}
	
	In this section, let $\Lambda$ be a finite dimensional algebra over an algebraically closed field $K$, and let $\mod\Lambda$ be the category of finitely presented right $\Lambda$-modules. We consider maximal green sequences in $\mod\Lambda$, which is an abelian length category.
	
	We first recall some basic notions of $\tau$-tilting theory. For a module $M$ in $\mod\Lambda$, let $|M|$ denote the number of nonisomorphic indecomposable summands of $M$. Let $\tau$ be the AR-translation. A module $M$ is called \textbf{$\tau$-rigid} if $\hom_{\Lambda}(M,\tau M)=0$.
	
	\begin{definition}\cite{AIR}
		Let $(M,P)$ be a pair with $M\in\mod\Lambda$ and $P\in\operatorname{proj}\Lambda$.
		\begin{enumerate}
			\item[(a)] The pair $(M,P)$ is called a \textbf{$\tau$-rigid pair} if $M$ is $\tau$-rigid and $\hom_{\Lambda}(P,M)=0$.
			\item[(b)] The pair $(M,P)$ is called a \textbf{support $\tau$-tilting pair} (resp. \textbf{almost complete support $\tau$-tilting pair}) if $(M,P)$ is $\tau$-rigid and $|M|+|P|=|\Lambda|$ (resp. $|M|+|P|=|\Lambda|-1$). In this case $M$ is called a \textbf{support $\tau$-tilting module} (resp. \textbf{almost complete support $\tau$-tilting module}).
		\end{enumerate}
	\end{definition}
	
	\begin{theorem}\cite{AIR}
		Let $\Lambda$ be a finite‑dimensional algebra. Then any basic almost complete support $\tau$-tilting pair for $\Lambda$ is a direct summand of exactly two basic support $\tau$-tilting pairs. These two pairs are called \textbf{mutations} of each other.
	\end{theorem}
	
	There is a bijection between support $\tau$-tilting modules and functorially finite torsion classes. Denote by $\ftors\Lambda$ the set of functorially finite torsion classes in $\mod\Lambda$ and by $\sstilt\Lambda$ the set of support $\tau$-tilting pairs.
	
	\begin{theorem}\cite{AIR}
		There is a bijection
		\begin{equation}
			\sstilt\Lambda \longleftrightarrow \ftors\Lambda
		\end{equation}
		given by $\sstilt\Lambda\ni T\mapsto\fac T\in\ftors\Lambda$ and $\ftors\Lambda\ni\cT\mapsto P(\cT)\in\sstilt\Lambda$, where $P(\cT)$ is the direct sum (up to isomorphism) of all indecomposable Ext‑projective objects in $\cT$. An object $X\in\cT$ is \emph{Ext‑projective} if $\ext_{\Lambda}^1(\cT,X)=0$.
	\end{theorem}
	
	This bijection allows mutations of torsion classes. The following theorem describes how mutations relate to inclusions of torsion classes.
	
	\begin{theorem}\cite[Theorem 3.1]{DIJ}\label{thm:mutation of torsion pair}
		Let $\Lambda$ be a finite dimensional algebra and $M$ a support $\tau$-tilting $\Lambda$-module. Then:
		\begin{enumerate}
			\item Let $\mathcal{T}$ be a torsion class in $\mod\Lambda$ such that $\fac M \supsetneq \mathcal{T}$. Then there exists $N\in\sstilt\Lambda$ with:
			\begin{itemize}
				\item $M$ and $N$ are mutations of each other,
				\item $\fac M \supsetneq \fac N \supset \mathcal{T}$.
			\end{itemize}
			\item Let $\mathcal{T}$ be a torsion class in $\mod\Lambda$ such that $\fac M \subsetneq \mathcal{T}$. Then there exists $L\in\sstilt\Lambda$ with:
			\begin{itemize}
				\item $M$ and $L$ are mutations of each other,
				\item $\fac M \subsetneq \fac L \subset \mathcal{T}$.
			\end{itemize}
		\end{enumerate}
	\end{theorem}
	
	We now examine the structure of maximal green sequences and prove a property regarding infinite length sequences.
	
	\begin{proposition}\label{prop:finitely generated torsion class}
		Let $\cS=\{\cT_{\alpha}\}_{\alpha\in J}$ be a maximal green sequence in $\mod\Lambda$.
		\begin{enumerate}
			\item If $\cS$ has infinite length, then not all torsion classes $\cT_{\alpha}$ can be finitely generated.
			\item If $\cS$ has finite length, then every torsion class $\cT_{\alpha}$ in $\cS$ is functorially finite.
		\end{enumerate}
	\end{proposition}
	
	\begin{proof}
		We first prove part (1). Set $\cT^0 := 0$ and $I^0 := \{i \in J \mid \cT^0 \subsetneq \cT_i\}$. For each $i \in I^0$, define
		\begin{equation}
			Q_i := \{\cT \mid \cT \text{ is a mutation of } \cT^0 \text{ and } \cT^0 \subsetneq \cT \subset \cT_i\}.
		\end{equation}
		By Theorem \ref{thm:mutation of torsion pair}, $Q_i$ is nonempty for each $i\in I^0$, and $Q_i\subset Q_{i'}$ whenever $i<i'$. Hence $\bigcap_{i\in I^0}Q_i$ is nonempty, because each $Q_i$ is finite and the decreasing chain stabilizes. Pick $\cT\in\bigcap_{i\in I^0}Q_i$; then $\cT^0\subsetneq\cT\subset\bigwedge\{\cT_i\mid i\in I^0\}$. Since $\cS$ is a maximal green sequence, $\cT^0$ is completely meet‑irreducible, so $\cT = \bigwedge\{\cT_i\mid i\in I^0\}$. Moreover, $\bigcap_{i\in I^0}Q_i$ contains exactly one torsion class.
		
		Let $\cT^1$ be the unique torsion class in $\bigcap_{i\in I^0}Q_i$; it is functorially finite. Replace $\cT^0$ by $\cT^1$ and repeat. By induction we obtain an increasing chain of functorially finite torsion classes $\{\cT^k\}_{k\in\mathbb{N}}$ contained in $\cS$, with $\cT^{k+1}$ covering $\cT^k$ in $\cS$ for each $k$. Thus $\{\cT^k\}$ is an ideal of $\cS$.
		
		Set $\cT:=\bigvee_{k\in\mathbb{N}}\cT^k$. Then $\cT$ is not completely join‑irreducible, and by Lemma \ref{lem:compact element} it is not compact in $\cS$, hence not compact in $\tors\Lambda$. Compact elements in $\tors\Lambda$ are exactly those of the form $\filt(\fac X)$ for some $X\in\mod\Lambda$ \cite[Proposition 3.2]{DIRRT}. Therefore $\cT$ is not finitely generated. Consequently, not all $\cT_{\alpha}$ can be finitely generated, which establishes (1).
		
		For part (2), assume $\cS$ has finite length. The construction above produces an increasing chain of functorially finite torsion classes. Because $\cS$ is finite, this chain must terminate after finitely many steps, reaching $\mod\Lambda$. Hence every torsion class $\cT_{\alpha}$ appearing in $\cS$ is functorially finite.
	\end{proof}
	
	\begin{example}
		Consider the Kronecker quiver
		\begin{equation}
			Q: \; 1 \rightrightarrows 2.
		\end{equation}
		Let $\Lambda = KQ$ be its path algebra over an algebraically closed field $K$. The category $\mod\Lambda$ of finite‑dimensional right $\Lambda$‑modules is an abelian length category. Its Auslander–Reiten quiver is well known and consists of three components: the preprojective component $\mathcal{P}$, the regular component $\mathcal{R}$ (a family of tubes of rank $1$), and the preinjective component $\mathcal{I}$. Denote by $P(i)$, $I(i)$ the indecomposable projective and injective modules on the vertex $i$ respectively. Then the preprojective component $\mathcal{P}$ has the shape
		\begin{equation}
			\begin{tikzcd}[
				column sep=1.12em,
				row sep=2em,
				cells={nodes={text width=2.5em, align=center, inner sep=1pt}}
				]
				& P_2 && P_4 && P_6 && {} \\
				P_1 && P_3 && P_5 && \cdots
				\arrow[shorten >=1pt, shorten <=1pt, from=1-2, to=2-3]
				\arrow[shorten >=1pt, shorten <=1pt, shift left=1.5, from=1-2, to=2-3]
				\arrow[shorten >=1pt, shorten <=1pt, from=1-4, to=2-5]
				\arrow[shorten >=1pt, shorten <=1pt, shift left=1.5, from=1-4, to=2-5]
				\arrow[shorten >=1pt, shorten <=1pt, from=1-6, to=2-7]
				\arrow[shorten >=1pt, shorten <=1pt, shift left=1.5, from=1-6, to=2-7]
				\arrow[shorten >=1pt, shorten <=1pt, from=2-1, to=1-2]
				\arrow[shorten >=1pt, shorten <=1pt, shift left=1.5, from=2-1, to=1-2]
				\arrow[shorten >=1pt, shorten <=1pt, from=2-3, to=1-4]
				\arrow[shorten >=1pt, shorten <=1pt, shift left=1.5, from=2-3, to=1-4]
				\arrow[shorten >=1pt, shorten <=1pt, from=2-5, to=1-6]
				\arrow[shorten >=1pt, shorten <=1pt, shift left=1.5, from=2-5, to=1-6]
			\end{tikzcd}
		\end{equation}
		where $P_{2k+1}=\tau^{-k}P(2)$ and $P_{2k+2}=\tau^{-k}P(1)$ for each $k\in\mathbb{Z}_{\ge 0}$. The preinjective component $\mathcal{I}$ has the shape
		\begin{equation}
			\begin{tikzcd}[
				column sep=1.12em,
				row sep=2em,
				cells={nodes={text width=2.5em, align=center, inner sep=1pt}}
				]
				\cdots && I_5 && I_3 && I_1 && {} \\
				& I_6 && I_4 && I_2
				\arrow[shorten >=1pt, shorten <=1pt, from=1-1, to=2-2]
				\arrow[shorten >=1pt, shorten <=1pt, shift left=1.5, from=1-1, to=2-2]
				\arrow[shorten >=1pt, shorten <=1pt, from=1-3, to=2-4]
				\arrow[shorten >=1pt, shorten <=1pt, shift left=1.5, from=1-3, to=2-4]
				\arrow[shorten >=1pt, shorten <=1pt, from=1-5, to=2-6]
				\arrow[shorten >=1pt, shorten <=1pt, shift left=1.5, from=1-5, to=2-6]
				\arrow[shorten >=1pt, shorten <=1pt, from=2-2, to=1-3]
				\arrow[shorten >=1pt, shorten <=1pt, shift left=1.5, from=2-2, to=1-3]
				\arrow[shorten >=1pt, shorten <=1pt, from=2-4, to=1-5]
				\arrow[shorten >=1pt, shorten <=1pt, shift left=1.5, from=2-4, to=1-5]
				\arrow[shorten >=1pt, shorten <=1pt, from=2-6, to=1-7]
				\arrow[shorten >=1pt, shorten <=1pt, shift left=1.5, from=2-6, to=1-7]
			\end{tikzcd}.
		\end{equation}
		where $I_{2k+1}=\tau^{k} I(1)$ and $I_{2k}=\tau^{k} I(2)$ for each $k\in \mathbb{Z}_{\ge 0}$ (with $k=0$ giving $I_1=I(1)$, $I_2=I(2)$). The regular part $\mathcal{R}$ consists of a family of pairwise orthogonal standard homogeneous tubes indexed by the projective line $\mathbb{P}_1(K)$. For each $[\lambda: \mu]\in \mathbb{P}_1(K)$, denote the corresponding tube by $\mathbb{T}_{[\lambda:\mu]}$. It contains a unique brick (up to isomorphism) and we denote it by $S_{[\lambda,\mu]}$, where
		\begin{equation}
			\begin{tikzcd}
				{S_{[\lambda,\mu]}=K} & K
				\arrow["\mu"', shift right=1, from=1-1, to=1-2]
				\arrow["\lambda", shift left=1, from=1-1, to=1-2]
			\end{tikzcd}
		\end{equation}
		
		(1) There is exactly one maximal green sequence of finite length, namely, $\cS=\{\cT_i\}_{i=0}^2$ where
		\begin{equation}
			\cT_0=\{0\}, \quad \cT_1=\add S_2,\quad \cT_2=\mod KQ.
		\end{equation}
		In this case, each torsion class is both completely join-irreducible and completely meet-irreducible in $\cS$. Its corresponding CBHO sequence is $(S(1),S(2))$. The finest stability data corresponding with $\cS$ is $(\Phi, \{\Pi_{\varphi}\}_{\varphi\in\Phi})$ where $\Phi=\{1,2\}$ and $\Pi_{1}=\add S_1$ and $\Pi_{2}=\add S_2$.
		
		(2) We now construct maximal green sequences of infinite length. Fix a total order $\preceq$ on the projective line $\mathbb{P}_1(K)$. Define the totally ordered set
		\begin{equation}
			\Phi = \{(0,n)\mid n\in\mathbb{Z}_{+}\}\;\sqcup\;\mathbb{P}_1(K)\;\sqcup\;\{(1,-n)\mid n\in\mathbb{Z}_{+}\},
		\end{equation}
		with the order
		\begin{equation}
			(0,1) < (0,2) < (0,3) < \cdots < [\lambda:\mu]\;(\text{ordered by }\preceq) < \cdots < (1,-3) < (1,-2) < (1,-1).
		\end{equation}
		For each $\varphi\in\Phi$, define a brick $B_\varphi$ by
		\begin{equation}
			B_\varphi = 
			\begin{cases}
				P_n, & \text{if } \varphi = (0,n),\; n\ge 1,\\[4pt]
				S_{[\lambda,\mu]}, & \text{if } \varphi = [\lambda:\mu]\in\mathbb{P}_1(K),\\[4pt]
				I_n, & \text{if } \varphi = (1,-n),\; n\ge 1.
			\end{cases}
		\end{equation}
		
		Then $\{B_\varphi\}_{\varphi\in\Phi}$ is a CBHO sequence in $\mod\Lambda$. By Theorem~\ref{thm:correspondence of three sets}, it corresponds to a finest stability data $(\Phi,\{\Pi_\varphi\}_{\varphi\in\Phi})$ with $\Pi_\varphi = \filt\{B_\varphi\}$, and hence to a maximal green sequence. Explicitly, for each filter $F$ of $\Phi$ (including $\emptyset$ and $\Phi$), the corresponding torsion class is 
		\begin{equation}
			\cT_F = \filt\{B_\varphi \mid \varphi\in F\}.
		\end{equation}
		Thus we obtain a maximal green sequence of infinite length for every total order $\preceq$ on $\mathbb{P}_1(K)$. 
		
		In this maximal green sequence $\{\cT_F\}$, a torsion class $\cT_F$ is completely join-irreducible if and only if $F$ is a principal filter of $\Phi$, i.e., $F = \uparrow \varphi_0$ for some $\varphi_0\in \Phi$. In such a case, $\uparrow \varphi_0$ covers exactly one filter, namely $\uparrow \varphi_0 \setminus \{\varphi_0\}$, and one has $\cT_{\uparrow \varphi_0} \cap \cF_{\uparrow \varphi_0 \setminus \{\varphi_0\}} = B_{\varphi_0}$.
	\end{example}

	\section{Connections with cluster theory}\label{sec:cluster}
	
	In this section, we consider cluster algebras of infinite rank. More precisely, we consider the cluster algebra associated with a so‑called strongly almost finite quiver following \cite{HL}, which is used to investigate the folding theory of sign‑skew‑symmetric cluster algebras. We will see that maximal green sequences of infinite length naturally arise in cluster algebras of infinite rank.
	
	\subsection{Maximal green sequences in cluster quiver pattern of infinite rank}
	We first recall some terminologies related to (infinite) quivers \cite{BLP}. Let $Q$ be an (infinite) quiver. We denote by $V(Q)$ the set of vertices of $Q$ and by $A(Q)$ the set of arrows of $Q$. For $i \in V(Q)$, denote by $i_{+}$ (respectively, $i_{-}$) the set of arrows starting at (respectively, ending at) $i$. The quiver $Q$ is called \textbf{locally finite} if $i_{+}$ and $i_{-}$ are finite for any $i \in V(Q)$. For $i, j \in V(Q)$, let $Q(i, j)$ denote the set of paths from $i$ to $j$ in $Q$. Then $Q$ is called \textbf{interval‑finite} if $Q(i, j)$ is finite for any $i, j \in V(Q)$, and it is called \textbf{strongly locally finite} if it is locally finite and interval‑finite. In this paper, we always assume that the vertex set $V(Q)$ is at most countable. In this study, we call $Q$ \textbf{strongly almost finite} if it is strongly locally finite and has no infinite paths.
	
	Sometimes the ``interval‑finite'' condition is not easy to check directly. Note that a strongly almost finite quiver $Q$ is always acyclic, as it is interval‑finite. (In particular, a strongly almost finite quiver is always a cluster quiver, i.e., it has no loops and no $2$-cycles.) The following lemma provides a convenient equivalent characterization.
	
	\begin{lemma}
		Suppose $Q$ is an acyclic locally finite quiver that does not admit infinite paths. Then $Q$ is strongly almost finite.
	\end{lemma}
	
	\begin{proof}
		It suffices to show that $Q$ is interval‑finite. First, we claim that for any pair of vertices $(i,j)$ and a fixed $n \in \mathbb{Z}_+$, the set $Q(i,j)$ contains only finitely many paths of length $n$. For $n=1$, the claim holds because $Q$ is locally finite. Any path in $Q(i,j)$ of length $n+1$ must pass through an immediate successor $i'$ of $i$. Since $i$ has only finitely many immediate successors, the claim follows by induction on $n$.
		
		Now suppose, if possible, that there exist vertices $i_0, j_0$ such that $|Q(i_0, j_0)| = \infty$. Consequently, there is no upper bound on the lengths of paths in $Q(i_0, j_0)$. Because $|(i_0)_+| < \infty$, we can find an immediate successor $i_1 \in (i_0)_+$ such that $|Q(i_1, j_0)| = \infty$ and again there is no upper bound on the lengths of paths in $Q(i_1, j_0)$. Repeating this argument inductively yields an infinite path passing through $i_0, i_1, \dots$, which contradicts the hypothesis that $Q$ admits no infinite paths.
	\end{proof}
	
	Similar to classical cluster theory, an \textbf{ice quiver} is a pair $(Q,F)$, where $Q$ is a cluster quiver (possibly infinite) and $F\subseteq V(Q)$ is a subset of vertices called the \textbf{frozen vertices}, such that there are no arrows between them.
	
	\begin{remark}
		For an ice quiver $R$, we denote by $V(R)$ the set of all vertices, by $V_m(R)$ the set of mutable (non‑frozen) vertices, and by $V_f(R)$ the set of frozen vertices. Thus $V(R)=V_m(R)\sqcup V_f(R)$ and there are no arrows among $V_f(R)$. We also denote by $A(R)$ the set of all arrows in $R$.
	\end{remark}
	
	Assume that $Q$ is a strongly almost finite quiver. We fix a copy $\bar{V}(Q)=\{\bar{i}\mid i\in V(Q)\}$ and identify $V(Q)=\mathbb{Z}_+$, $\bar{V}(Q)=\bar{\mathbb{Z}}_+=\{\bar{i}\mid i\in\mathbb{Z}_+\}$. For a subset $S\subseteq V(Q)$, denote by $Q^S$ the full subquiver on $S$. For $s\in\mathbb{Z}_+$, set $[s]:=\{1,\dots,s\}$, $[\bar{s}]:=\{\bar{1},\dots,\bar{s}\}$, and their complements $[s]^c:=\{s+1,s+2,\dots\}$, $[\bar{s}]^c:=\{\overline{s+1},\overline{s+2},\dots\}$.
	
	\begin{definition}
		Suppose $Q$ is a strongly almost finite quiver. The \textbf{framed quiver} associated with $Q$ is the ice quiver $\hat{Q}$ defined as follows:
		\begin{align}
			V(\hat{Q}) &= V(Q)\sqcup \bar{V}(Q), \\
			A(\hat{Q}) &= A(Q) \sqcup \{ i \gets \bar{i} \mid i \in \mathbb{Z}_+ \}.
		\end{align}
		The vertices $\{\bar{i}: i\in\mathbb{Z}_+\}$ are frozen.  
		Dually, the \textbf{coframed quiver} associated with $Q$ is the ice quiver $\check{Q}$ defined as follows:
		\begin{align}
			V(\check{Q}) &= V(Q)\sqcup \bar{V}(Q), \\
			A(\check{Q}) &= A(Q) \sqcup \{ i \to \bar{i} \mid i \in \mathbb{Z}_+ \},
		\end{align}
		with the same set of frozen vertices.
	\end{definition}
	
	Now we define the mutation of ice quivers, which does not involve cluster variables.
	
	\begin{definition}[Mutation of an ice quiver]
		Let $R$ be an ice quiver with the vertex set $V(R)=V_m(R)\sqcup V_f(R)$. For a mutable vertex $k \in V_m(R)$, the \textbf{mutation} $\mu_k(R)$ is a new ice quiver obtained from $R$ by the following three steps:
		\begin{enumerate}
			\item For every path $a \to k \to b$ where $a, b\in V(R)$, add an arrow $a \to b$;
			\item Reverse each arrow incident to $k$;
			\item Remove any resulting $2$-cycles.
		\end{enumerate}
		The set of frozen vertices remains unchanged.
	\end{definition}
	
	\begin{remark}
		If $R$ is locally finite, then only finitely many paths $a \to k \to b$ exist; thus each mutation changes only finitely many arrows and $\mu_k(R)$ remains locally finite. However, even when $R$ is strongly almost finite, $\mu_k(R)$ does not need to be strongly almost finite, because it may no longer be acyclic.
	\end{remark}
	
	For classical cluster theory, we refer to \cite{FZ1}, \cite{FZ2} and \cite{HL-Chinese}. Now we forget the cluster variables and focus solely on the combinatorics of quiver mutations.
	
	\begin{definition}\label{def:cluster quiver pattern}
		Let $\mathbb{T}$ be the $|\mathbb{Z}_+|$‑regular tree whose edges are labeled by positive integers.  A \textbf{cluster quiver pattern} on the initial ice quiver $\hat{Q}$ is an assignment of an ice quiver $R_t$ to each vertex $t$ of $\mathbb{T}$ such that:
		\begin{enumerate}
			\item $R_{t_0} = \hat{Q}$ for a fixed root vertex $t_0$;
			\item For any edge $t \stackrel{k}{\longrightarrow} t'$ (labeled by $k \in \mathbb{Z}_+$), the quiver $R_{t'}$ is obtained from $R_t$ by applying the mutation $\mu_k$ as defined above.
		\end{enumerate}
	\end{definition}
	
	By definition, all quivers appearing in the pattern are mutation equivalent to $\hat{Q}$. Therefore, for any $t \in \mathbb{T}$, $R_t$ can be regarded as the initial ice quiver, giving rise to the same cluster quiver pattern. If we follow a (finite or infinite) walk in $\mathbb{T}$ starting from $t_0$, we obtain a sequence of mutations.
	
	From now on, suppose $Q$ is a strongly almost finite quiver. The following definitions classify mutable vertices according to the arrows they have with frozen vertices.
	
	\begin{definition}
		Let $R\in \mut(\hat{Q})$. A vertex $i\in V_m(R)$ is called \textbf{green} if
		\begin{equation}
			\{\bar{j}\in \bar{V}(Q)\mid \exists\  \bar{j} \leftarrow i \in A(R)\}=\emptyset;
		\end{equation}
		dually, it is called \textbf{red} if
		\begin{equation}
			\{\bar{j}\in \bar{V}(Q)\mid \exists\  \bar{j} \rightarrow i \in A(R)\}=\emptyset.
		\end{equation}
	\end{definition}
	
	For each mutable vertex $i\in V_m(R)$, we now define its \textbf{$\mathbf{c}$-vector} $c_i$ as a column vector (indexed by the frozen vertices) whose $j$‑th entry is the integer
	\begin{equation}\label{eq:c-vectors}
		c_{ij} \;=\; \#\{\bar{j} \to i \text{ in } A(R)\} \;-\; \#\{i \to \bar{j} \text{ in } A(R)\},
		\qquad \bar{j} \in V_f(R).
	\end{equation}
	Because $R$ is locally finite, both terms in the right hand side of \eqref{eq:c-vectors} are finite. Since there are no $2$-cycles, these two terms cannot both be non‑zero. 
	
	The \textbf{$\mathbf{c}$-matrix} of $R$ is defined as the matrix $(c_1,c_2\cdots )$ consisting of all $\mathbf{c}$-vectors (each regarded as a column).
	
	Consider the mutation of $R$ at direction $k$. The $\mathbf{c}$-matrix of $\mu_k(R)$ is $(c'_1,c'_2,\cdots)$, where the entries are given by
	\begin{equation}
		c'_{ij} = 
		\begin{cases}
			-c_{ik}, & j = k,\\[6pt]
			c_{ij} + \operatorname{sign}(c_{ik})\,\bigl[\,c_{ik}\,b_{kj}\bigr]_+, & j \neq k,
		\end{cases}
	\end{equation}
	with $b_{kj} = \#\{k\to j\} - \#\{j\to k\}$, $[x]_+ = \max(0,x)$, and $\operatorname{sign}(0)=0$.
	
	For the initial framed quiver $\hat{Q}$, by definition, the $\mathbf{c}$-vector of vertex $i$ is the column vector that has $+1$ in the $i$‑th coordinate and $0$ elsewhere. Thus the $\mathbf{c}$-matrix of $\hat{Q}$ is the infinite identity matrix $I_\infty$. Dually, the $\mathbf{c}$-matrix of $\check{Q}$ is $-I_\infty$.
	
	Note that $i$ is green exactly when all entries of its $\mathbf{c}$-vector are $\le 0$, and red exactly when all entries are $\ge 0$. Hence sign‑coherence of $\mathbf{c}$-vectors (all entries of a $\mathbf{c}$-vector are either all non‑positive or all non‑negative) is equivalent to the statement that every vertex in $V_m(R)$ is either green or red.
	
	For a finite quiver $Q$, it is a classical fact that in every quiver mutation‑equivalent to $\hat{Q}$, each mutable vertex is either green or red (i. e. the sign‑coherence of $\mathbf{c}$-vectors, see\cite{DWZ1}). Indeed, this fact can be easily extended to the case of strongly almost finite quiver, which is as follows:
	
	\begin{proposition}\label{prop:sign-coherence of strongly almost finite quiver}
		Suppose $Q$ is a strongly almost finite quiver and $R$ is an ice quiver mutation‑equivalent to $\hat{Q}$.  Then every vertex in $V_m(R)$ is either green or red.  Equivalently, each $\mathbf{c}$-vector of $R$ is sign‑coherent.  Moreover, for each $i\in V_m(R)$, only finitely many arrows connect it to frozen vertices; hence each $\mathbf{c}$-vector has only finitely many non‑zero entries.
	\end{proposition}
	
	\begin{proof}
		Observe that $\hat{Q}$ itself has the property: each mutable vertex $i$ has exactly one incoming arrow from the frozen vertex $\bar{i}$ and no other arrows to/from frozen vertices; hence it is green and its $\mathbf{c}$-vector is the function that is $+1$ at $j=i$ and $0$ elsewhere.  Because $R$ is obtained by a (possibly infinite) sequence of mutations, we argue by finite approximations.
		
		Since $Q$ is locally finite, any finite initial segment of mutations only affects a finite subquiver.  More precisely, for any finite set of mutable vertices $\{i_1,\dots,i_m\}$, there exists a finite subquiver $Q^{[M]}$ of $Q$ (on vertices $1,\dots,M$) large enough so that mutating along $\{i_1,\dots,i_m\}$ inside $\hat{Q}$ and then restricting to the first $M$ vertices yields the same result as first restricting to $\hat{Q}^{[M]}$ and then performing the mutations.  The finite‑rank theory \cite{DWZ1} tells us that in the finite quiver $\hat{Q}^{[M]}$, every mutable vertex in the mutated seed is either green or red (i.e., its $\mathbf{c}$-vector is sign‑coherent).  This property passes to the limit because the color (or the sign of the $\mathbf{c}$-vector entries) of a fixed vertex stabilizes after finitely many mutations.  The finiteness of arrows to frozen vertices follows from the local finiteness of the mutation process.
	\end{proof}
	
	Thanks to Proposition \ref{prop:sign-coherence of strongly almost finite quiver}, the following definition is well‑defined.
	
	\begin{definition}
		Suppose $Q$ is a strongly almost finite quiver.  A \textbf{green sequence} for $Q$ is a (possibly infinite) sequence $\mathbf{k}=(k_1,k_2,\dots)\subset V(Q)$ such that $k_1$ is green in $\hat{Q}$ and for any $k\ge 2$, the vertex $k_i$ is green in $\mu_{k_{i-1}}\circ\cdots\circ\mu_{k_1}(\hat{Q})$.
	\end{definition}
	
	In the language of the cluster quiver pattern introduced in Definition \ref{def:cluster quiver pattern}, the definition above coincides exactly with the usual notion of a green sequence in cluster algebras in the case that $Q$ is a finite quiver.
	
	As is well known, two quivers $Q$ and $Q'$ (finite or infinite) are \textbf{isomorphic} (as quivers) if there exists a bijection $\sigma: Q_0 \to Q_0'$ such that for all $i,j\in Q_0$,
	\begin{equation}
		\#\{i\to j \text{ in } Q\} = \#\{\sigma(i)\to \sigma(j) \text{ in } Q'\}.
	\end{equation}
	
	Now suppose $R$ and $R'$ are two ice quivers with $V_m(R)=V_m(R')$ and $V_f(R)=V_f(R')$. We call $R$ and $R'$ are \textbf{ice isomorphic} if there exists a quiver isomorphism $\sigma$ (as above) from $R$ to $R'$ such that $\sigma(v)=v$ for every frozen vertex $v\in V_f(R)$.
	
	For a finite cluster quiver $Q$, a sequence $\mathbf{k}=(k_1,\dots,k_l)$ is called a \textbf{maximal green sequence} if every vertex in $V_m(\mu_{\mathbf{k}}(\hat{Q}))$ is red, where $\mu_{\mathbf{k}}(\hat{Q})=\mu_{k_{l}}\circ\cdots\circ\mu_{k_1}(\hat{Q})$. Furthermore, suppose $R\in \mut(\hat{Q})$. If all vertices in $V_m(R)$ are green (respectively, red), then $R$ is ice isomorphic to $\hat{Q}$ (respectively, $R$ is ice isomorphic to $\check{Q}$) \cite[Proposition 2.10]{BDP}.
	
	In classical cluster theory, maximal green sequences are always of finite length. However, for the cluster algebra of infinite rank, equivalently, the cluster quiver pattern from a strongly almost finite quiver, one cannot expect a maximal green sequence to be finite, because a single mutation can turn at most finitely many green vertices to red. Motivated by this observation, we introduce the following notion of convergence for sequences of quivers.
	
	\begin{definition}\label{def:convergence of quiver}
		Let $\{Q_n\}_{n=1}^{+\infty}$ be a sequence of ice quivers and $Q_\infty$ an ice quiver, where each $Q_n$ and $Q_\infty$ are strongly almost finite and share the same unfrozen vertex set $\mathbb{Z}_+$ and the same frozen vertex set $\bar{\mathbb{Z}}_+$. We say that the sequence \textbf{converges} to $Q_\infty$, and write
		\begin{equation}
			\lim_{n\to +\infty} Q_n = Q_\infty,
		\end{equation}
		if for every positive integer $s$, there exists a positive integer $N$ such that for all $n\ge N$, the full subquiver $Q_n^{[s]\sqcup[\bar{s}]}$ equals $Q_\infty^{[s]\sqcup[\bar{s}]}$.
	\end{definition}
	
	\begin{proposition}
		Using the notation established above, if the limit $\lim\limits_{n\to\infty} Q_n$ exists, then it is unique.
	\end{proposition}
	
	\begin{proof}
		Suppose $\lim\limits_{n\to\infty} Q_n = Q_\infty$ and $\lim\limits_{n\to\infty} Q_n = Q_\infty'$.  For any vertices $i,j$ (frozen or unfrozen), choose $s$ sufficiently large so that $Q_n^{[s]\sqcup[\bar{s}]}$ contains both $i$ and $j$. By convergence, there exists $N$ such that for all $n\ge N$, the full subquivers $Q_n^{[s]\sqcup[\bar{s}]}$, $Q_\infty^{[s]\sqcup[\bar{s}]}$ and $Q_\infty'^{[s]\sqcup[\bar{s}]}$ are equal. In particular, the number of arrows $i\to j$ in $Q_\infty$ equals that in $Q_\infty'$. Hence $Q_\infty = Q_\infty'$.
	\end{proof}
	
	Now we can define maximal green sequences for cluster algebras of infinite rank.
	
	\begin{definition}\label{def:mgs-inf-rank}
		Suppose $Q$ is a strongly almost finite quiver.  Let $\mathbf{k}=(k_1,k_2,\cdots)$ be an infinite sequence of unfrozen vertices and set $Q_n = \mu_{k_n}\circ\cdots\circ\mu_{k_1}(\hat{Q})$ for each $n\ge1$.  We call that $\mathbf{k}$ is a \textbf{maximal green sequence} for $Q$ if the following conditions are satisfied:
		
		(1) The limit $\lim\limits_{n\to\infty} Q_n$ exists (in the sense of Definition~\ref{def:convergence of quiver});
		
		(2) The limiting ice quiver, denoted by $Q_\infty$, is ice isomorphic to the coframed quiver $\check{Q}$. That is to say, there exists an isomorphism $\sigma:Q_{\infty}\to \check{Q}$ such that for each frozen vertex $\bar{i}\in \bar{\bZ}_+$, $\sigma(\bar{i})=\bar{i}$.
	\end{definition}
	
	\begin{remark}
		For finite cluster algebras, any maximal green sequence from $\hat{Q}$ to $\check{Q}$ can be reversed to obtain a sequence from $\check{Q}$ to $\hat{Q}$.  In infinite rank, reversal of an infinite sequence is not directly possible.  One could define a \textbf{maximal green co‑sequence} as an infinite sequence $\mathbf{k}'$ such that $\lim\limits_{n\to\infty} \mu_{k'_n}\circ\cdots\circ\mu_{k'_1}(\check{Q}) \cong \hat{Q}$.  However, this is essentially the same as a maximal green sequence for the opposite quiver $Q^{\mathrm{op}}$, so we restrict attention to maximal green sequences as defined above.
	\end{remark}
	
	In what follows, we will construct a class of maximal green sequence of a strongly almost finite quiver.
	
	Suppose $Q$ is a quiver and $Q_1$, $Q_2$ are two full subquivers of $Q$. We say that $Q$ is a \textbf{triangular extension} of $Q_1$ by $Q_2$ if the set of vertices of $Q$ is a disjoint union of the sets of $Q_1$ and $Q_2$ and there are no arrows from $Q_1$ to $Q_2$. Cao and Li have proved the following theorem for the existence of maximal green sequences of a finite quiver (Cao and Li actually proved the case of skew‑symmetrizable cluster algebras; see \cite{CL}. Here we only discuss quivers, i.e., the skew‑symmetric case.)
	
	\begin{theorem}\label{thm:triangular extension and mgs}\cite[Theorem 4.5]{CL}
		Suppose $Q$ is a finite quiver and is a triangular extension of $Q_2$ by $Q_1$. Let $\mathbf{k}_1:=(k_{11},\cdots ,k_{1 t_1})$ and $\mathbf{k}_2:=(k_{21},\cdots ,k_{2 t_2})$ be mutation sequences of $Q_1$ and $Q_2$ respectively. Then $\mathbf{k}:=(k_{11},\cdots ,k_{1 t_1}; k_{21},\cdots ,k_{2 t_2})$ is a maximal green sequence of $Q$ if and only if $\mathbf{k}_1$ and $\mathbf{k}_2$ are maximal green sequences of $Q_1$ and $Q_2$ respectively.
	\end{theorem}
	
	We will generalize this theorem to the strongly almost finite quivers.
	
	\begin{definition}
		Suppose $Q$ is a strongly almost finite quiver and $\{Q_p\}_{p=1}^\infty$ is a sequence of finite quiver. We say that $Q$ is a \textbf{triangular extension} of $\{Q_p\}_{p=1}^{\infty}$ if the set of vertices of $Q$ is a disjoint union of the set of each $Q_p$, each $Q_p$ is a full subquiver of $Q$ and there are no arrows from vertices of $Q_q$ to vertices of $Q_p$ whenever $q<p$.
	\end{definition}
	
	To prove the main theorem of this subsection, we need the following lemma.
	
	\begin{lemma}\label{lem:equivalent condition on sign-coherence}
		Let $Q$ be a strongly almost finite quiver and $Q_1$ , $Q_2$ are two full subquivers of $Q$ where $Q_1$ is finite. Suppose $Q$ is a triangular extension of $Q_2$ by $Q_1$. Then after any sequence of mutations within $Q_1$, there will never appear a path of the form $j_1 \to i \to j_2$ with $i \in V(Q_1)$ and $j_1, j_2 \in V(Q_2)$. Moreover, this property is equivalent to that after any sequence of mutations within $Q_1$, the full subquiver $Q_2$ remains invariant.
	\end{lemma}
	
	The proof of this lemma is similar to that of \cite[Corollary 3.3, Proposition 3.7]{CL}, since $Q$ is strongly almost finite and a finite sequence of mutations allows us to restrict to a finite subquiver of $Q$. With Lemma~\ref{lem:equivalent condition on sign-coherence} at hand, we can now prove the following generalization of Theorem~\ref{thm:triangular extension and mgs} to strongly almost finite quivers.
	
	We make the following convention: for two finite mutation sequences $\mathbf{k}=(k_1,\dots,k_m)$ and $\mathbf{k}'=(k'_1,\dots,k'_n)$, define $\mathbf{k}\ast\mathbf{k}'=(k_1,\dots,k_m,k'_1,\dots,k'_n)$.
	
	\begin{theorem}\label{thm:maximal green sequence and triangular extension}
		Suppose $Q$ is a strongly almost finite quiver and is a triangular extension of the sequence of finite quivers $\{Q_p\}_{p=1}^\infty$. For each $p\ge 1$, let $t_p$ be a positive integer and let 
		\begin{equation}
			\mathbf{k}_p = (k_{p,1}, k_{p,2}, \dots, k_{p,t_p})
		\end{equation}
		be a maximal green sequence of $Q_p$. Then
		\begin{equation}
			\mathbf{k} = \mathbf{k}_1 \ast \mathbf{k}_2 \ast \cdots \ast \mathbf{k}_p\ast\ \cdots.
		\end{equation}
		is a maximal green sequence of $Q$.
	\end{theorem}
	
	\begin{proof}
		First note that $V_m(\hat{Q}) = V(Q) = \bigsqcup_{p \ge 1} V(Q_p)$ and $V_f(\hat{Q}) = \bar{V}(Q) = \bigsqcup_{p \ge 1} \bar{V}(Q_p)$.  
		
		For each $p \ge 1$, set  
		\begin{equation}
			Q^{(p)} := \mu_{\mathbf{k}_p} \circ \cdots \circ \mu_{\mathbf{k}_1}(\hat{Q}).
		\end{equation}
		We prove by induction on $p$ the following two properties:
		
		\begin{itemize}
			\item[(A)] The full subquiver of $Q^{(p)}$ generated by
			$\bigsqcup_{q \le p} \bigl( V(Q_q) \sqcup \bar{V}(Q_q) \bigr)$ is ice-isomorphic to 
			$\check{Q}_{\le p}$, the coframed quiver of the full subquiver 
			$Q_{\le p} \subset Q$ on $\bigsqcup_{q \le p} V(Q_q)$. The isomorphism is given by permutations 
			$\sigma_1, \dots, \sigma_p$ obtained from the maximal green sequences of $Q_1, \dots, Q_p$.
			
			\item[(B)] $Q^{(p)}$ (viewed as an ordinary quiver) is a triangular extension of the complement of $Q_{p+1}$ by $Q_{p+1}$ (both viewed as a subquiver of $Q^{(p)}$). Moreover, when applying mutation sequence $\mathbf{k}_p$ on $Q^{(p-1)}$, the subquiver of $Q^{(p-1)}$ generated by 
			$V(\hat{Q}) \setminus V(Q_p)$ remains invariant at each step. 
		\end{itemize}
		
		For $p=1$, $\hat{Q}$ is a triangular extension of the complement of $Q_1$ by $Q_1$. Since $\mathbf{k}_1$ is a maximal green sequence of $Q_1$, the full subquiver of $Q^{(1)}$ generated by $V(Q_1)\sqcup \bar{V}(Q_1)$ is ice‑isomorphic to $\check{Q}_1$ via some permutation $\sigma_1$ and thus (A) holds. By Lemma~\ref{lem:equivalent condition on sign-coherence}, no arrows appear from the complement of $V(Q_1)$ to $V(Q_1)$ in $Q^{(1)}=\mu_{\mathbf{k}_1}\hat{Q}$. Therefore $Q^{(1)}$ is a triangular extension of the complement of $Q_2$ by $Q_2$. Again by Lemma~\ref{lem:equivalent condition on sign-coherence}, when applying $\mathbf{k}_1$ on $\hat{Q}$, the subquiver of $\hat{Q}$ generated by $V(\hat{Q})\setminus V(Q_1)$ remains invariant at each step and thus (B) holds.
		
		Assume that properties (A) and (B) hold for $p$; we will prove them for $p+1$. Since $\mathbf{k}_{p+1}$ is a maximal green sequence of $Q_{p+1}$, the full subquiver $Q^{(p+1)}$ generated by $V(Q_{p+1})\sqcup \bar{V}(Q_{p+1})$ is ice‑isomorphic to $\check{Q}_{p+1}$ via some permutation $\sigma_{p+1}$. For the same reason with the case $p=1$, property (B) holds for $p+1$. Now we consider the mutation sequence $\mathbf{k}_1 \ast \mathbf{k}_2 \ast \cdots$ on the framed quiver $\hat{Q}_{\leq p+1}$. Then every vertex in $V_m(\mu_{\mathbf{k}_{p+1}}\circ\cdots\circ\mu_{\mathbf{k}_1} \hat{Q}_{\leq p+1})$ is red and thus $\mathbf{k}_1 \ast \mathbf{k}_2 \ast \cdots$ is a maximal green sequence of $Q_{\leq {p+1}}$. Therefore by \cite[Proposition 2.10]{BDP}, the property (A) holds for $p+1$.
		
		The following structural property of the mutated quivers will be established.
		
		\begin{corollary}\label{cor:no arrows between blocks}
			Under the assumptions of Theorem~\ref{thm:maximal green sequence and triangular extension}, for any two distinct positive integers $p$ and $q$, at every step of the mutation sequence $\mathbf{k}$ applied to the framed quiver $\hat{Q}$, the resulting quiver contains no arrow between any vertex of $V(Q_p)$ and any vertex of $\bar{V}(Q_q)$.
		\end{corollary}
		
		\begin{proof}
			The statement holds trivially for the initial quiver $\hat{Q}$ because by construction there are no arrows between $V(Q_p)$ and $\bar{V}(Q_q)$ for $p\neq q$. Now consider an arbitrary step of the mutation sequence $\mathbf{k}$. When we mutate at a vertex belonging to some $V(Q_p)$, two observations are crucial:
			\begin{itemize}
				\item By Lemma~\ref{lem:equivalent condition on sign-coherence}, no path of the form $i \to k \to j$ can appear with $i,k\in V(Q_p)$ and $j\in V(Q_q)$ for $q\neq p$ (or symmetrically with $k,j\in V(Q_p)$ and $i\in V(Q_q)$). Hence mutations inside $Q_p$ never create an arrow between $V(Q_p)$ and $\bar{V}(Q_q)$.
				\item Property (B) of the induction ensures that the full subquiver on the complement of $Q_p$ (i.e. all vertices not in $V(Q_p)$) stays invariant throughout the mutations within $Q_p$.
			\end{itemize}
			Combining these facts with the mutation rule for quivers and an induction on the number of steps, we conclude that no arrow between $V(Q_p)$ and $\bar{V}(Q_q)$ ever appears for $p\neq q$ at any stage of $\mathbf{k}$. 
		\end{proof}
		
		Now we return to the proof of Theorem~\ref{thm:maximal green sequence and triangular extension}.
		
		Clearly, since each $\mathbf{k}_p$ is a maximal green sequence of $Q_p$, it follows from properties (A) and (B) that $\mathbf{k}$ is a green sequence of $Q$. Consequently, one readily checks from Definition~\ref{def:convergence of quiver} and Definition~\ref{def:mgs-inf-rank} that the sequence of quivers obtained from $\hat{Q}$ by mutating along $\mathbf{k}$ converges, and the limit is ice‑isomorphic to $\check{Q}$ via the global permutation obtained by combining all $\sigma_p$.
		
		We conclude that $\mathbf{k}$ is a maximal green sequence of $Q$.
	\end{proof}
	
	As an application of Theorem~\ref{thm:maximal green sequence and triangular extension}, we now construct two explicit maximal green sequences for the alternating $\mathbb{A}_{\infty}$ quiver.
	
	\begin{example}
		Consider the framed quiver of the alternating $\mathbb{A}_{\infty}$ quiver:
		\begin{equation}
			\begin{tikzcd}
				\bar{1} & \bar{2} & \bar{3} & \bar{4} & \bar{5} & \bar{6} & \\
				1 & 2 & 3 & 4 & 5 & 6 & \cdots
				\arrow[from=1-1, to=2-1]
				\arrow[from=1-2, to=2-2]
				\arrow[from=1-3, to=2-3]
				\arrow[from=1-4, to=2-4]
				\arrow[from=1-5, to=2-5]
				\arrow[from=1-6, to=2-6]
				\arrow[from=2-1, to=2-2]
				\arrow[from=2-3, to=2-2]
				\arrow[from=2-3, to=2-4]
				\arrow[from=2-5, to=2-4]
				\arrow[from=2-5, to=2-6]
				\arrow[from=2-7, to=2-6]
			\end{tikzcd}
		\end{equation}
		This quiver is a triangular extension of the family $\{Q_p\}_{p=1}^\infty$, where each $Q_p$ is isomorphic to the $\mathbb{A}_2$ quiver $1 \rightarrow 2$. For $\mathbb{A}_2$ it is well known that $(2,1)$ is a maximal green sequence; the corresponding permutation $\sigma_p$ (as in Definition~\ref{def:mgs-inf-rank}) is the identity on $\{1,2\}$. Hence by Theorem~\ref{thm:maximal green sequence and triangular extension}, the sequence
		\[
		\mathbf{k}= (2,1;\; 4,3;\; 6,5;\; \dots)
		\]
		is a maximal green sequence of $\mathbb{A}_{\infty}$. The limiting ice quiver is isomorphic to the coframed quiver $\check{Q}$ via the global permutation $\sigma$ which acts as the identity on every pair $\{2n-1,2n\}$.
		
		A second maximal green sequence for $\mathbb{A}_2$ is given by $(1,2,1)$. In this case the induced permutation $\sigma_p$ on the vertices of $Q_p$ is the transposition $(1,2)$. Applying Theorem~\ref{thm:maximal green sequence and triangular extension} to the infinite sequence
		\[
		\mathbf{k}= (1,2,1;\; 3,4,3;\; 5,6,5;\; \dots)
		\]
		yields a maximal green sequence of $\mathbb{A}_{\infty}$. The resulting limit quiver $Q_\infty$ is
		\begin{equation}
			\begin{tikzcd}
				\bar{1} & \bar{2} & \bar{3} & \bar{4} & \bar{5} & \bar{6} & \\
				2 & 1 & 4 & 3 & 6 & 5 & \cdots
				\arrow[from=2-1, to=1-1]
				\arrow[from=2-1, to=2-2]
				\arrow[from=2-2, to=1-2]
				\arrow[from=2-3, to=1-3]
				\arrow[from=2-3, to=2-2]
				\arrow[from=2-3, to=2-4]
				\arrow[from=2-4, to=1-4]
				\arrow[from=2-5, to=1-5]
				\arrow[from=2-5, to=2-4]
				\arrow[from=2-5, to=2-6]
				\arrow[from=2-6, to=1-6]
				\arrow[from=2-7, to=2-6]
			\end{tikzcd}
		\end{equation}
		and it is ice‑isomorphic to $\check{Q}$ via the global permutation $\sigma$ that swaps $2n-1$ and $2n$ for every $n\ge 1$; that is, $\sigma = (1,2)(3,4)(5,6)\cdots$.
	\end{example}
	
	In the next section, we will prove the converse of Theorem~\ref{thm:maximal green sequence and triangular extension} using categorification methods.

	\subsection{Categorification of maximal green sequences of infinite length}
	
	Maximal green sequences are intimately related to the representation theory of quivers. In the finite setting, they correspond to sequences of mutations of $\tau$-tilting objects and to chains of torsion classes in the module category. In this section we aim to categorify the infinite length maximal green sequences constructed above, by establishing analogous connections for strongly almost finite quivers.
	
	Suppose $Q$ is a strongly almost finite quiver and $KQ$ is its path algebra over an algebraically closed field $K$. For a vertex $a\in V(Q)$, let $e_a$ be the corresponding primitive idempotent. Define the following right $KQ$-modules:
	\begin{equation}
		P(a):=e_a KQ, \qquad I(a):=D(KQ e_a), \quad \text{where } D(-)=\hom_{K}(-,K).
	\end{equation}
	Let $\proj KQ$ (respectively, $\inj KQ$) be the category of finite direct sums of modules of the form $P(a)$ (respectively, $I(a)$). Let $S(a):=P(a)/\rad(P(a))$.
	
	A $KQ$-module $M$ is called \textbf{finitely presented} if there exist $P_1, P_0\in\proj KQ$ such that the following sequence is exact:
	\begin{equation}
		P_1 \longrightarrow P_0 \longrightarrow M \longrightarrow 0.
	\end{equation}
	Dually, $M$ is called \textbf{finitely co-presented} if there exist $I_0, I_1\in\inj KQ$ such that the following sequence is exact:
	\begin{equation}
		0 \longrightarrow M \longrightarrow I_0 \longrightarrow I_1.
	\end{equation}
	
	According to \cite[Proposition 1.15]{BLP}, for a $KQ$-module $M$, since $Q$ does not have infinite paths, the following conditions are equivalent:
	\begin{enumerate}
		\item[(i)] $M$ is finitely presented;
		\item[(ii)] $M$ is finite-dimensional (as a K-vector space);
		\item[(iii)] $M$ is finitely co-presented;
		\item[(iv)] $M$ has finite length.
	\end{enumerate}
	\begin{remark}
		If $Q$ is a finite quiver, the equivalence above is well-known and each condition is also equivalent to $M$ being finitely generated.
	\end{remark}
	
	Therefore, we denote by $\mod KQ$ the category of finitely presented right $KQ$ modules. It is well-known that $\mod KQ$ is a hereditary abelian category. Furthermore, as shown in \cite[Proposition 1.16]{BLP}, the modules $P(a)$ are exactly the indecomposable projective objects in $\mod KQ$, the modules $I(a)$ are exactly the indecomposable injective objects in $\mod KQ$, and the modules $S(a)$ are exactly the simple objects in $\mod KQ$.
	
	Suppose $Q$ is a strongly almost finite quiver and its vertex set is denoted by $\left\{1,2,\cdots\right\}$. For each $m\in \mathbb{Z}_+$, we denote by $Q^{[m]}$ the full subquiver of $Q$ which is generated by the vertices $\{1,2,\cdots, m\}$. Therefore, we have an exact inclusion functor $F_i^j:\mod KQ^{[i]}\rightarrow \mod KQ^{[j]}$ for each $i\leq j$ satisfying 
	\begin{enumerate}
		\item[(a)] $F_i^i = \operatorname{id}_{\mod KQ^{[i]}}$;
		\item[(b)] $F_j^k \circ F_i^j = F_i^k$ for $i \le j \le k$.
	\end{enumerate}
	Therefore, by \cite[Lemma 4.2, Remark 4.3]{HL}, the \textbf{colimit} of $(\mod KQ^{[i]}, F_i^j)$ exists and is isomorphic to $\mod KQ$ as an abelian category. We denote it by $\varinjlim \mod KQ^{[m]} = \mod KQ$. More precisely,
	\begin{enumerate}
		\item[(1)] There exist exact functors (indeed, the inclusion functors) $F_i:\mod KQ^{[i]}\rightarrow \mod KQ$ for all $i\in \mathbb{Z}_+$ and natural equivalence $\eta_i^j:F_i\rightarrow F_j\circ F_i^j$ for all $i\leq j$ satisfying $\eta_i^i=\operatorname{id}_{F_i}$, and the following diagrams
		\begin{equation}
			\begin{tikzcd}
				{F_i} & {F_j\circ F_i^j} \\
				{F_i} & {F_k\circ F_i^k}
				\arrow["{\eta_i^j}", from=1-1, to=1-2]
				\arrow["{\operatorname{id}_{F_i}}"', from=1-1, to=2-1]
				\arrow["{\eta_j^k*\operatorname{id}_{F_i^j}}", from=1-2, to=2-2]
				\arrow["{\eta_i^k}"', from=2-1, to=2-2]
			\end{tikzcd}
		\end{equation}
		commute for any $i\leq j\leq k$ where ``$*$'' means the horizontal composition of functors;
		
		\item[(2)] For any abelian category $\cA$ with exact functors $G_i: \mod KQ^{[i]}\rightarrow \cA$ and natural equivalence $\xi_i^j:G_i\rightarrow G_j\circ F_i^j$ for all $i\leq j$ satisfying $\xi_i^i=\operatorname{id}_{G_i}$ and the following commutative diagrams
		\begin{equation}
			\begin{tikzcd}
				{G_i} & {G_j\circ F_i^j} \\
				{G_i} & {G_k\circ F_i^k}
				\arrow["{\xi_i^j}", from=1-1, to=1-2]
				\arrow["{\operatorname{id}_{G_i}}"', from=1-1, to=2-1]
				\arrow["{\xi_j^k*\operatorname{id}_{F_i^j}}", from=1-2, to=2-2]
				\arrow["{\xi_i^k}"', from=2-1, to=2-2]
			\end{tikzcd}
		\end{equation}
		for all $i\leq j\leq k$, there exists a unique (up to natural equivalence) exact functor $H:\mod KQ\rightarrow \cA$ such that $G_i=H\circ F_i$ for all $i\in \mathbb{Z}_+$ and $\xi_i^j=\operatorname{id}_H * \eta_i^j$ for all $i\leq j$.
	\end{enumerate}
	
	Now we give the main theorem of this section, which is a categorification of maximal green sequences of cluster algebras of infinite rank.
	
	\begin{theorem}\label{thm:categorification of mgs}
		Let $Q$ be a strongly almost finite quiver, and let $\mathbf{k}=(k_1,k_2,\dots)\subset V(Q)$ be a maximal green sequence of $Q$ (see Definition~\ref{def:mgs-inf-rank}) which satisfies the following two conditions:
		\begin{enumerate}
			\item The permutation $\sigma$ appearing in Definition~\ref{def:mgs-inf-rank} has no infinite cycle.
			\item For any $m\in \mathbb{Z}_+$, there exists $N\in \mathbb{Z}_+$ such that for all $n>N$, the quiver $\mu_{k_n}\circ\cdots\circ\mu_{k_1}(\hat{Q})$ contains no arrow between $[\bar{m}]$ and $[m]^c$.
		\end{enumerate}
		Then there exists a unique maximal green sequence
		\begin{equation}
			\{\mathcal{T}_n \mid n\in \mathbb{Z}_{\ge 0}\cup\{\infty\}\}
		\end{equation}
		of torsion classes in $\mod KQ$ with the following properties:
		\begin{enumerate}
			\item[(a)] For each $1\leq n< \infty$, there exists a unique rigid brick $B_n$ in $\mod KQ$ such that:
			\begin{equation}
				\mathcal{T}_{n-1}^{\perp}\cap \mathcal{T}_n = \operatorname{add} B_n,
			\end{equation}
			
			\item[(b)] The dimension vector of $B_n$ equals the $\mathbf{c}$-vector on the vertex $k_n$ in the mutated quiver $\mu_{k_{n-1}}\circ\cdots\circ\mu_{k_1}(\hat{Q})$.
		\end{enumerate}
	\end{theorem}
	
	\begin{proof}
		We divide the proof into four steps.
		
		\textbf{Step 1. Restriction to finite subquivers.}
		
		For each $m\ge1$, let $Q^{[m]}$ be the full subquiver of $Q$ on vertices $\{1,\dots,m\}$; it is a finite quiver. Let $\hat{Q}^{[m]\sqcup[\bar{m}]}$ be the full subquiver of the framed quiver $\hat{Q}$ on the vertex set $[m]\sqcup[\bar{m}]$. Fix an integer $n_0\ge1$ and consider the truncated mutation sequence $\mathbf{k}|_{n_0}=(k_1,\dots,k_{n_0})$. 
		Since $Q$ is locally finite, the mutations in $\mathbf{k}|_{n_0}$ affect only finitely many arrows. Hence we can choose $m$ large enough so that
		\begin{equation}
			\bigl(\mu_{k_{n_0}}\circ\cdots\circ\mu_{k_1}(\hat{Q})\bigr)^{[m]\sqcup[\bar{m}]}
			= \mu_{k_{n_0}}\circ\cdots\circ\mu_{k_1}\bigl(\hat{Q}^{[m]\sqcup[\bar{m}]}\bigr).
		\end{equation}
		In other words, mutating the whole infinite quiver $\hat{Q}$ and then restricting to the vertex set $[m]\sqcup[\bar{m}]$ coincides with first restricting to the same vertex set and then performing the same mutations.
		
		\textbf{Step 2. Construction of the torsion classes on the finite piece.}
		
		Apply the finite‑rank correspondence \cite[Theorem 5.1]{DK} to the quiver $Q^{[m]}$. We obtain an ascending chain of torsion classes in $\mod KQ^{[m]}$,
		\begin{equation}
			0=\cT_0\subset \cT_1\subset\cdots\subset \cT_{n_0},
		\end{equation}
		such that for each $1\le n\le n_0$,
		\begin{equation}\label{eq:brick-equation}
			\cT_{n-1}^{\perp}\cap \cT_n = \add B_n^{(m)},
		\end{equation}
		where each $B_n^{(m)}$ is a rigid brick in $\mod KQ^{[m]}$. 
		Moreover, the dimension vector of $B_n^{(m)}$ coincides with the $\mathbf{c}$-vector on the vertex $k_n$ in the quiver $\mu_{k_{n-1}}\circ\cdots\circ\mu_{k_1}(\hat{Q})$.
		
		Let $F_m:\mod KQ^{[m]}\to\mod KQ$ be the inclusion functor. 
		Since $F_m$ is exact, each $F_m(\mathcal{T}_n)$ is still a torsion class in $\operatorname{mod} KQ$. By construction, we have $F_m(\mathcal{T}_n) = \filt \{B_t^{(m)} \mid 1\le t\le n\}$ for $1\le n\le n_0$ (see Theorem~\ref{thm:correspondence of three sets}).
		
		\textbf{Step 3. Passing to the infinite sequence.}
		
		Note that for fixed $n$, the bricks $B^{(m)}_n$ for different $m$ are isomorphic in $\mod  KQ$ when $m$ is sufficiently large, and we will simply denote any of them by $B_n$. Similarly, the torsion classes $F_m(\cT_n)$ for different $m$ are equivalent when $m$ is sufficiently large; for simplicity, we still denote them by $\cT_n$. Letting $n_0$ (in step 2) run through all positive integers, we obtain an infinite ascending chain of torsion classes
		\begin{equation}
			\{\cT_n\mid n\in\mathbb{Z}_+\}
		\end{equation}
		in $\mod KQ$ that satisfies \eqref{eq:brick-equation} for every $n\ge1$.
		
		\textbf{Step 4. Proving $\bigvee \{\cT_n\mid n\in \mathbb{Z}_+ \} = \mod KQ$.}
		
		Set $Q_n := \mu_{k_n}\circ\cdots\circ\mu_{k_1}(\hat{Q})$. By Definition~\ref{def:mgs-inf-rank} the limit $Q_\infty := \lim\limits_{n\to\infty}Q_n$ exists and is ice‑isomorphic to $\check{Q}$ via a permutation $\sigma:\mathbb{Z}_+\to\mathbb{Z}_+$ which permutes all the unfrozen vertices.
		
		Fix an arbitrary indecomposable projective module $P(a)$ (corresponding to vertex $t$ of $Q$). 
		Since $P(a)$ is finitely supported and $\sigma$ has no infinite cycle, we can choose a finite set $\{1,\dots,m\}$ that contains the support of $P(a)$ and is a union of (finite) orbits of $\sigma$.
		
		By our assumption (2), there exists $N\in \mathbb{Z}_+$ such that for all $n>N$, $(Q_n)^{[m]\sqcup[\bar{m}]}=(Q_\infty)^{[m]\sqcup[\bar{m}]}$ holds and $Q_n$ contains no arrow between $[\bar{m}]$ and $[m]^c$. Consequently, for $n\ge N$ there is no arrow between any frozen vertex $\bar{i}$ with $i\le m$ and any unfrozen vertex $j$ with $j>m$.
		
		Now consider the infinite quiver $Q_N = \mu_{k_N}\circ\cdots\circ\mu_{k_1}(\hat{Q})$. Using the local finiteness of $Q$, we can choose $m'\ge m$ such that
		\begin{equation}
			\mu_{k_N}\circ\cdots\circ\mu_{k_1}(Q^{[m]}) = (Q_N)^{[m]},
		\end{equation}
		and such that for every $1\le n\le N$, no arrow exists between the unfrozen vertex $k_n$ and any frozen vertex $\bar{j}$ with $j>m'$ (this is possible because the mutation sequence is finite and only finitely many arrows are created).
		
		It follows that the support of each brick $B_n$ ($1\le n\le N$) is contained in $\{1,\dots,m'\}$. Hence each torsion class $\cT_n = \filt\{B_1,\dots,B_n\}$ is contained in $\mod KQ^{[m']}$. Thus $\{\cT_n\mid 1\le n\le N\}$ is a finite green sequence in $\mod KQ^{[m']}$, and each $\cT_n$ is functorially finite. Let $(M_N,P_N)$ be the support $\tau$-tilting pair corresponding to $\cT_N$; i.e., $M_N$ is the direct sum of all indecomposable $\ext$-projective objects in $\cT_N$, and $P_N$ is the unique basic projective object in $\mod KQ^{[m']}$ with $\hom(P_N,M_N)=0$.
		
		Now we analyze the $\mathbf{c}$-matrix of the pair $(M_N,P_N)$ using the quiver $Q_N$. Because $Q_N$ has no arrows between $[\bar{m}]$ and $[m]^c$. The $\mathbf{c}$-matrix $C_N$ of the pair $(M_N, P_N)$ has a block‑triangular form
		\begin{equation}
			C_N = \begin{pmatrix}
				-P^{-1} & 0 \\
				\ast   & \ast
			\end{pmatrix},
		\end{equation}
		where $P$ is the $m\times m$ permutation matrix induced by the restriction of $\sigma$ to $\{1,\dots,m\}$.
		By the tropical duality between $\mathbf{c}$-matrices and $\mathbf{g}$-matrices (see e.g. \cite{Fu}; note that our definition differs from that in \cite{Fu} by a sign, which accounts for the negative sign), the $\mathbf{g}$-matrix $G_N = -(C_N^{-1})^{T}$ is of the shape
		\begin{equation}
			G_N = \begin{pmatrix}
				P^{T} & \ast \\
				0     & \ast
			\end{pmatrix}.
		\end{equation}
		The first $m$ columns are $m$ $\mathbf{g}$-vectors of the support $\tau$-tilting pair $(M_N, P_N)$ and thus $P(a)$ is a direct summand of $M_N$. Therefore
		\begin{equation}
			P(a) \in \fac(M_N)= \cT_N.
		\end{equation}
		Since $t$ was arbitrary, each indecomposable projective module (and thus each module) belongs to $\bigvee \{\cT_n\mid n\in \mathbb{Z}_+ \}$. Note that each indecomposable object in $\mod KQ$ is a factor of the direct sum of some $P(a)$ and $\cT_n$ is closed under factors, we conclude that 
		\begin{equation}
			\bigvee \{\cT_n\mid n\in \mathbb{Z}_+ \} = \mod KQ
		\end{equation} 
		This completes the proof.
	\end{proof}
	
	\begin{remark}
		(1) We remark that the two conditions in Theorem~\ref{thm:categorification of mgs} are not restrictive. Indeed, by Theorem~\ref{thm:maximal green sequence and triangular extension} together with Corollary~\ref{cor:no arrows between blocks}, one can construct many maximal green sequences that satisfy both conditions.
		
		(2) Unlike the finite case, the difficulty in the proof of Theorem~\ref{thm:categorification of mgs} lies in Step~4; the first three steps are straightforward, involving restriction to finite subquivers and an ascending chain of torsion classes satisfying the brick equation. Step~4 uses this restriction to finite subquivers and the $\mathbf{c}$-matrix to show each indecomposable projective module lies in some $\mathcal{T}_N$---an argument that requires analyzing the limit of quivers, which is unnecessary in the finite case.
	\end{remark}
	
	Now we prove the converse of Theorem~\ref{thm:maximal green sequence and triangular extension} using categorification methods. This approach allows us, under certain conditions, to decompose an infinite-length maximal green sequence into a sequence of finite-length maximal green sequences.
	
	\begin{proposition}\label{prop:decompose-mgs}
		Suppose $Q$ is a strongly almost finite quiver and is a triangular extension of the sequence of finite quivers $\{Q_p\}_{p=1}^\infty$. Suppose
		\begin{equation}
			\mathbf{k}=\mathbf{k}_1 \ast \mathbf{k}_2 \ast \cdots
		\end{equation}
		is a maximal green sequence of $Q$ satisfying conditions (1) and (2) in Theorem~\ref{thm:categorification of mgs}, where $\mathbf{k}_p=(k_{p1},\dots,k_{p t_p})$ is a mutation sequence of $Q_p$ for each $p$. Then each $\mathbf{k}_p$ is a maximal green sequence of $Q_p$.
	\end{proposition}
	
	\begin{proof}
		By our assumption and Theorem~\ref{thm:categorification of mgs}, there exists a unique maximal green sequence
		\begin{equation}
			\{\cT_n \mid n\in \bZ_{\ge 0}\cup\{\infty\}\}
		\end{equation}
		of torsion classes in $\mod KQ$ with the following properties:
		
		(a) For each $1\leq n< \infty$,
		\begin{equation}
			\cT_{n-1}^{\perp}\cap \cT_n = \operatorname{add} B_n,
		\end{equation}
		where $B_n$ is a unique rigid brick in $\mod KQ$.
		
		(b) The dimension vector of $B_n$ equals the $\mathbf{c}$-vector on the vertex $k_n$ in the mutated quiver $\mu_{k_{n-1}}\circ\cdots\circ\mu_{k_1}(\hat{Q})$.
		
		Now we fix $p\in \bZ_+$. Note that $Q_p$ is a full subquiver of $Q$, thus we can view $\mod KQ_p$ as a subcategory of $\mod KQ$ via the inclusion functor. Let $\cT_n'=\cT_n\cap \mod KQ_p$ for each $n\in \bZ_{\ge 0}\cup \{\infty\}$. In particular, $\cT_0'=\{0\}$ and $\cT_{\infty}'=\mod KQ_p$.
		
		We claim that for each $n\in \bZ_{\ge 0}$, $\cT_{n-1}'\subsetneq\cT_n'$ if and only if $B_n\in \mod KQ_p$. Indeed, the ``if'' part is clear. For the ``only if'' part, suppose, if possible, there is a torsion class $\cT'\in\mod KQ_p$ such that $\cT_{n-1}'\subsetneq \cT'\subsetneq \cT_{n}'$, then we have $\cT_{n-1}\subsetneq \cT'\vee \cT_{n-1}\subsetneq \cT_{n}$, which is a contradiction. Therefore, $\cT_{n}'$ covers $\cT_{n-1}'$ in $\mod KQ_p$ and hence $\cT_{n}'=\filt(\cT_{n-1}'\cup \{B_{n}'\})$ for some brick $B_{n}'\in \mod KQ_p$. Thus $\cT_{n-1}\subsetneq \cT_{n-1}\vee\cT_{n}'\subset \cT_n$. Since $\cT_{n-1}\vee\cT_{n}'=\filt(\cT_{n-1}\cup \{B_n'\})$, we have $\cT_n=\filt(\cT_{n-1}\cup \{B_n'\})$. By the uniqueness of $B_n$, the ``only if'' part holds and we have proved the claim.
		
		Therefore, after removing duplicate terms from the chain of torsion classes $\{\cT_{n}'\mid n\in \bZ_{\ge 0}\cup\{\infty\}\}$, we obtain a maximal green sequence $\cS$ in $\mod KQ_p$. By our construction, each $\cT_n'$ is finitely generated in $\mod KQ_p$. Then by Proposition~\ref{prop:finitely generated torsion class}, we know that $\cS$ is of finite length. Therefore, by the categorification of maximal green sequences in the finite quiver case (see \cite[Theorem~5.3]{DK}) --- which gives a bijection between maximal green sequences of a quiver and that of torsion classes, the chain $\{\mathcal{T}_n'\}$ obtained in $\mod KQ_p$ is a maximal chain of torsion classes, hence corresponds uniquely to a maximal green sequence. By construction, this sequence is $\mathbf{k}_p$. Thus $\mathbf{k}_p$ is a maximal green sequence of $Q_p$.
	\end{proof}
	
	As a direct consequence of Theorem~\ref{thm:maximal green sequence and triangular extension} and Proposition~\ref{prop:decompose-mgs}, we obtain the following characterization.
	
	\begin{corollary}\label{cor:triangulation and mgs}
		Let $Q$ be a strongly almost finite quiver which is a triangular extension of the sequence of finite quivers $\{Q_p\}_{p=1}^\infty$. Let $\mathbf{k} = \mathbf{k}_1 \ast \mathbf{k}_2 \ast \cdots$ be a mutation sequence of $Q$ where each $\mathbf{k}_p$ is a sequence of vertices of $Q_p$. Assume that $\mathbf{k}$ satisfies conditions (1) and (2) of Theorem~\ref{thm:categorification of mgs}. Then $\mathbf{k}$ is a maximal green sequence of $Q$ if and only if each $\mathbf{k}_p$ is a maximal green sequence of $Q_p$.
	\end{corollary}
	
	Finally, we will give an example of the categorification of maximal green sequence of the cluster algebra of infinite rank.
	
	\begin{example}
		We consider the alternating $\mathbb{A}_{\infty}$ quiver as in the previous example. 
		
		\begin{equation}
			Q=
			\begin{tikzcd}
				1 & 2 & 3 & 4 & 5 & 6 & \cdots
				\arrow[from=1-1, to=1-2]
				\arrow[from=1-3, to=1-2]
				\arrow[from=1-3, to=1-4]
				\arrow[from=1-5, to=1-4]
				\arrow[from=1-5, to=1-6]
				\arrow[from=1-7, to=1-6]
			\end{tikzcd}
		\end{equation}
		
		(1) Take the maximal green sequence 
		\begin{equation}
			\mathbf{k} = (2,1;\; 4,3;\; 6,5;\; \dots)
		\end{equation}
		constructed above. The associated permutation $\sigma$ is the identity, which has no infinite cycles. By Theorem~\ref{thm:categorification of mgs}, there exists a unique chain of torsion classes $\{\cT_n\}_{n\ge0}$ in $\mod KQ$ satisfying properties (a) and (b) of the theorem, with $\cT_0=0$ and $\bigvee_n \cT_n = \mod KQ$. 
		For each $p\ge1$, after mutating the block $Q_p = \{2p-1,2p\}$ we obtain two torsion classes. More explicitly, for the first block the mutations $(2,1)$ yield
		\begin{equation}
			\cT_1 = \filt(S(2)),\quad \cT_2 = \filt(S(2), S(1))=\mod KQ^{[2]},
		\end{equation}
		where $S(i)$ denotes the simple module at vertex $i$. The corresponding rigid bricks are $B_1=S(2)$ and $B_2=S(1)$. Their $\mathbf{c}$-vectors are respectively $e_2$ (i.e. the vector with $1$ in position $2$ and $0$ elsewhere) and $e_1$. After the second block we get
		\begin{equation}
			\cT_3 = \cT_2 \vee \filt(S(4)),\quad 
			\cT_4 = \cT_3 \vee \filt(S(4)\oplus S(3))=\mod KQ^{[4]},
		\end{equation}
		and so on. Thus the whole chain consists of the torsion classes obtained by successively adding the simple modules in the order $2,1,4,3,6,5,\dots$, and each $\mathbf{c}$-vector of the brick is the corresponding standard basis vector.
		
		(2) Consider the sequence 
		\[
		\mathbf{k} = (1,2,1;\; 3,4,3;\; 5,6,5;\; \dots).
		\]
		The permutation $\sigma$ is $(1,2)(3,4)\cdots$, whose orbits are all of length $2$, hence no infinite cycles. By Theorem~\ref{thm:categorification of mgs}, we again obtain a unique chain of torsion classes $\{\cT_n\}_{n\ge0}$ in $\mod KQ$ satisfying (a) and (b), with $\cT_0=0$ and $\bigvee_n \cT_n = \mod KQ$. For each block $Q_p$, the three mutations induce three torsion classes. For the first block, they are
		\begin{equation}
			\cT_1 = \filt(S(1)),\quad
			\cT_2 = \filt(S(1), P(1)),\quad
			\cT_3 = \filt(S(1), P(1), S(2))=\mod KQ^{[2]},
		\end{equation}
		where $P(1)$ is the projective indecomposable module at vertex $1$. The corresponding bricks are $S(1)$, $S(2)$ and $P(1)$; note that $P(1)$ is indecomposable and rigid, hence a brick. Their $\mathbf{c}$-vectors are respectively $e_1$, $e_2$, and $e_1+e_2$ (i. e., $(1,1,0,0,\dots)$). The subsequent blocks behave similarly, with analogous bricks and $\mathbf{c}$-vectors, and the whole chain converges to the whole module category.
		
		Thus both sequences satisfy the conditions of Theorem~\ref{thm:categorification of mgs} and provide explicit maximal green sequences of torsion classes for the alternating $\mathbb{A}_{\infty}$ quiver.
	\end{example}

    \vspace{\baselineskip}
    \textbf{Acknowledgments}\quad This work is supported by the National Natural Science Foundation of China (No. 12131015).
	
	\bibliographystyle{plain}
	\bibliography{references}  

@article{AIR,
	title = {{$\tau$}-tilting theory},
	author = {Adachi, T. and Iyama, O. and Reiten, I.},
	journal = {Compositio Mathematica},
	volume = {150},
	number = {3},
	pages = {415--452},
	year = {2014}
}

@article{Br,
	title = {Stability conditions on triangulated categories},
	author = {Bridgeland, T.},
	journal = {Annals of Mathematics},
	pages = {317--345},
	year = {2007}
}

@article{BCZ,
	title = {Minimal inclusions of torsion classes},
	author = {Barnard, E. and Carroll, A. and Zhu, S.},
	journal = {Algebraic Combinatorics},
	volume = {2},
	number = {5},
	pages = {879--901},
	year = {2019}
}

@article{BDP,
	title = {On maximal green sequences},
	author = {Br{\"u}stle, T. and Dupont, G. and P{\'e}rotin, M.},
	journal = {International Mathematics Research Notices},
	volume = {2014},
	number = {16},
	pages = {4547--4586},
	year = {2014}
}

@article{BLP,
	title = {Representation theory of strongly locally finite quivers},
	author = {Bautista, R. and Liu, S. and Paquette, C.},
	journal = {Proceedings of the London Mathematical Society},
	volume = {106},
	number = {1},
	pages = {97--162},
	year = {2013}
}

@article{BST2,
	title = {Stability conditions and maximal green sequences in abelian categories},
	author = {Br{\"u}stle, T. and Smith, D. and Treffinger, H.},
	journal = {Revista de la Uni{\'o}n Matem{\'a}tica Argentina},
	volume = {63},
	number = {1},
	pages = {203--221},
	year = {2022}
}

@article{CL,
	title = {Uniform column sign-coherence and the existence of maximal green sequences},
	author = {Cao, P. and Li, F.},
	journal = {Journal of Algebraic Combinatorics},
	volume = {50},
	number = {4},
	pages = {403--417},
	year = {2019}
}

@article{CLR,
	title = {Stability approach to torsion pairs on abelian categories},
	author = {Chen, M. and Lin, Y. and Ruan, S.},
	journal = {Journal of Algebra},
	volume = {636},
	pages = {560--602},
	year = {2023}
}

@article{DIJ,
	title = {{$\tau$}-tilting finite algebras, bricks, and $g$-vectors},
	author = {Demonet, L. and Iyama, O. and Jasso, G.},
	journal = {International Mathematics Research Notices},
	volume = {2019},
	number = {3},
	pages = {852--892},
	year = {2019}
}

@article{DIRRT,
	title = {Lattice theory of torsion classes: beyond {$\tau$}-tilting theory},
	author = {Demonet, L. and Iyama, O. and Reading, N. and Reiten, I. and Thomas, H.},
	journal = {Transactions of the American Mathematical Society, Series B},
	volume = {10},
	number = {18},
	pages = {542--612},
	year = {2023}
}

@article{DK,
	title = {A survey on maximal green sequences},
	author = {Demonet, L. and Keller, B.},
	journal = {arXiv preprint arXiv:1904.09247},
	year = {2019}
}

@book{DP,
	title = {Introduction to lattices and order},
	author = {Davey, B. A. and Priestley, H. A.},
	publisher = {Cambridge University Press},
	year = {2002}
}

@article{DWZ1,
	title = {Quivers with potentials and their representations {I}: Mutations},
	author = {Derksen, H. and Weyman, J. and Zelevinsky, A.},
	journal = {Selecta Mathematica},
	volume = {14},
	number = {1},
	pages = {59--119},
	year = {2008}
}

@article{En,
	title = {Schur's lemma for exact categories implies abelian},
	author = {Enomoto, H.},
	journal = {Journal of Algebra},
	volume = {584},
	pages = {260--269},
	year = {2021}
}

@article{Fu,
	title = {$c$-vectors via $\tau$-tilting theory},
	author = {Fu, C.},
	journal = {Journal of Algebra},
	volume = {473},
	pages = {194--220},
	year = {2017}
}

@article{FZ1,
	title = {Cluster algebras {I}: foundations},
	author = {Fomin, S. and Zelevinsky, A.},
	journal = {Journal of the American Mathematical Society},
	volume = {15},
	number = {2},
	pages = {497--529},
	year = {2002}
}

@article{FZ2,
	title = {Cluster algebras {II}: Finite type classification},
	author = {Fomin, S. and Zelevinsky, A.},
	journal = {arXiv preprint math/0208229},
	year = {2002}
}

@article{G,
	title = {Cluster algebras of infinite rank as colimits},
	author = {Gratz, S.},
	journal = {Mathematische Zeitschrift},
	volume = {281},
	number = {3},
	pages = {1137--1169},
	year = {2015}
}

@article{GKR,
	title = {$t$-stabilities and $t$-structures on triangulated categories},
	author = {Gorodentsev, A. L. and Kuleshov, S. A. and Rudakov, A. N.},
	journal = {Izvestiya: Mathematics},
	volume = {68},
	number = {4},
	pages = {749--781},
	year = {2004}
}

@article{HL,
	title = {Unfolding of sign-skew-symmetric cluster algebras and its applications to positivity and $F$-polynomials},
	author = {Huang, M. and Li, F.},
	journal = {Advances in Mathematics},
	volume = {340},
	pages = {221--283},
	year = {2018}
}

@book{HL-Chinese,
	title = {An Introduction to the Theory of Cluster Algebras \textnormal{(in Chinese)}},
	author = {Li, F. and Huang, M.},
	series = {the Basic Series of Modern Mathematics},
	number = {199},
	publisher = {Science Press},
	address = {Beijing},
	year = {2023},
	isbn = {978-7-03-074894-2}
}

@article{Ig1,
	title = {Maximal green sequences for cluster-tilted algebras of finite representation type},
	author = {Igusa, K.},
	journal = {Algebraic Combinatorics},
	volume = {2},
	number = {5},
	pages = {753--780},
	year = {2019}
}

@article{Ig2,
	title = {Linearity of stability conditions},
	author = {Igusa, K.},
	journal = {Communications in Algebra},
	volume = {48},
	number = {4},
	pages = {1671--1696},
	year = {2020}
}

@incollection{K,
	title = {On cluster theory and quantum dilogarithm identities},
	author = {Keller, B.},
	booktitle = {Representations of algebras and related topics},
	pages = {85--116},
	publisher = {European Mathematical Society-EMS-Publishing House GmbH},
	year = {2011}
}

@article{Ki,
	title = {Moduli of representations of finite dimensional algebras},
	author = {King, A. D.},
	journal = {The Quarterly Journal of Mathematics},
	volume = {45},
	number = {4},
	pages = {515--530},
	year = {1994}
}

@article{LL,
	title = {On maximal green sequences in abelian length categories},
	author = {Liu, S. and Li, F.},
	journal = {Journal of Algebra},
	volume = {580},
	pages = {399--422},
	year = {2021}
}

@book{Mu,
	title = {Geometric invariant theory},
	author = {Mumford, D. and Fogarty, J. and Kirwan, F.},
	volume = {34},
	publisher = {Springer},
	year = {1994}
}

@article{Ru,
	title = {Stability for an abelian category},
	author = {Rudakov, A.},
	journal = {Journal of Algebra},
	volume = {197},
	number = {1},
	pages = {231--245},
	year = {1997}
}

@article{Tr,
	title = {An algebraic approach to {H}arder–{N}arasimhan filtrations},
	author = {Treffinger, H.},
	journal = {Journal of Pure and Applied Algebra},
	volume = {229},
	number = {1},
	pages = {107817},
	year = {2025}
}
\end{document}